\documentclass[12pt]{amsart}
\usepackage{lscape,amssymb, amsmath,verbatim,graphicx,color}
\usepackage{epsfig,subfigure,float}
\usepackage{graphicx}
\usepackage{epsfig}
\usepackage{siunitx}
\usepackage{multirow}
\usepackage{adjustbox}

\numberwithin{figure}{section}
\newcommand\bve{\mbox{\boldmath${ \varepsilon}$}}
\newcommand\bW{{\bf W}}

\usepackage{rotating}

\usepackage{tikz}

\usepackage{setspace}

\newcommand{\lf}{\lfloor}
\newcommand{\rf}{\rfloor}

\newtheorem{rem}{Remark}[section]

\numberwithin{figure}{section}
\newtheorem{theorem}{Theorem}[section]
\newtheorem{lemma}{Lemma}[section]

\newtheorem{assumption}{Assumption}[section]
\newtheorem{condition}{Condition}[section]
\usepackage{bm}
\newcommand{\vare}{\varepsilon}

\newcommand\bbe{\mbox{\boldmath${ \beta}$}}

\newcommand\btheta{\mbox{\boldmath${ \theta}$}}
\newcommand\bTheta{\mbox{\boldmath${ \Theta}$}}

\newcommand\bX{{\bf X}}

\newcommand\T{\top}

\newcommand\bC{{\bf C}}

\newcommand\bx{{\bf x}}

\newcommand\bZ{{\bf Z}}

\def\beq{\begin{equation}}
\def\eeq{\end{equation}}
\def\bals{\begin{align*}}
\def\eals{\end{align*}}

\def\bal{\begin{align}}
\def\eal{\end{align}}

\numberwithin{equation}{section}
\numberwithin{theorem}{section}
\numberwithin{table}{section}

\begin{document}

\title[R\'enyi change point statistics]{A new class of change point test statistics of R\'enyi type }

\author {Lajos Horv\'ath}

\author{Curtis Miller}

\author{Gregory Rice}



\begin{abstract}
A new class of change point test statistics is proposed that utilizes a weighting and trimming scheme for the cumulative sum (CUSUM) process inspired by R\'enyi (1953). A thorough asymptotic analysis and simulations both demonstrate that this new class of statistics possess superior power compared to traditional change point statistics based on the CUSUM process when the change point is near the beginning or end of the sample. Generalizations of these ``R\'enyi" statistics are also developed to test for changes in the parameters in linear and non-linear regression models, and in generalized method of moments estimation. In these contexts we applied the proposed statistics, as well as several others, to test for changes in the coefficients of Fama-French factor models. We observed that the R\'enyi statistic was the most effective in terms of retrospectively detecting change points that occur near the endpoints of the sample.
\end{abstract}

\maketitle

\section{Introduction }\label{sec-int}

We consider in this paper the development of a new class of change point test statistics that are useful in addressing the problem of retrospectively detecting change points in parameters of interest that might occur near the beginning or end of a sequence of observations. Such end-of-sample change points are usually difficult to detect with traditional procedures, like those based on measuring the fluctuations of the cumulative sum (CUSUM) process, since they rely on the existence of a proportional number of observations before or after the change point in order to maintain power. For an up to date survey on change point analysis in the context of time series data, we refer the reader to Aue and Horv\'ath (2013).

Several authors have thus considered the problem of improving the power of change point tests in the case when the change might lie near the end points of the sample. The resulting methods typically resort to weighting the maximally selected CUSUM process or Wald F-type statistics, and this is sometimes combined with trimming the domain on which such processes are maximized. In the context of a single change in the mean of scalar time series, Andrews (1993) considers the maximum of the standardized CUSUM process on a trimmed domain, which corresponds to the maximally selected likelihood ratio test with independent and identically distributed normal observations. The standardized CUSUM process can also be maximized over all possible change points without trimming, and in this case a Darling--Erd\H{o}s result is needed in order to establish the null asymptotics of the statistic. The details and history of these results are chronicled in  Cs\"org\H{o} and Horv\'ath (1993), and Cs\"org\H{o} and Horv\'ath (1997). Darling--Erd\H{o}s type statistics are considered in Ling (2007) and Hidalgo and Seo (2013) in order to test for changes in the model parameters in general time series models with weakly dependent errors. In related work, Andrews (2003) develops end-of-sample instability tests assuming that the location of the potential change point in the parameters is known a priori, and that the number of observations before or after the change point is presumed to be fixed in the asymptotics. Some of these results and practical considerations for end-of-sample change point detection with known and unknown potential change points are surveyed in Busetti (2015).

We propose here a new class of change point test statistics that utilize more heavily weighted and trimmed CUSUM processes than have been considered previously. The inspiration for this weighting/trimming scheme comes from some results of R\'enyi related to the uniform empirical process, which were further studied by  Cs\"org\H{o}; see R\'enyi (1953) and Cs\"org\H{o} (1965). For this reason, we refer to the proposed statistics as ``R\'enyi" statistics. We establish the null asymptotics of these R\'enyi statistics below, which utilize some recent results on the rate of convergence in the Gaussian approximation of partial sum processes for general weakly dependent time series. We further show that these statistics possess superior power when compared to traditional tests based on the CUSUM process when the change point is near the end points of the sample. These results may also be generalized to test for changes in model parameters in linear and non-linear regression using the least squares principle and generalized method of moment estimation. In a Monte Carlo simulation study, we observed that the proposed statistic outperformed traditional CUSUM based statistics in this setting. We also demonstrate the proposed method in an application to test for change points in the parameters of the Fama-French five factor model. In this case we observed that the proposed statistic is more efficient in retrospectively detecting evident change points when they are near the end points of the sample.

The paper is organized as follows. In Section \ref{sec-main}, the R\'enyi statistic is introduced and we develop its asymptotic properties. This statistic is extended to test for changes in linear and nonlinear regression in Section \ref{ext}. An asymptotic comparison between the proposed test and standard CUSUM based test statistics is given in Section \ref{sec-alt}. We present the results of a Monte Carlo simulation study in which we compare the proposed statistic to several competitors for end of sample changes in Section \ref{simu}, and then present a data application in Section \ref{appl}. Proofs of all results are given in an online supplement to this article Horv\'ath et al. (2018+). The code and data used to reproduce the simulation study and data analysis are available at {\tt{ https://github.com/ntguardian/CPAT}}.

\section{R\'enyi change point statistics, and their null asymptotics }\label{sec-main}

At first, we restrict our attention to the simple location model for scalar random variables
 \begin{align}\label{model-1}
 X_t=\mu_t+e_t,\quad 1\leq t \leq T,
 \end{align}
where $Ee_i=0$, and further consider the hypothesis test of
$
H_0: \mu_1=\mu_2=\ldots =\mu_T,
$
versus the at most one change in the mean alternative,
$
H_A: \mu_1=\mu_2=\ldots =\mu_{t^*}\neq \mu_{t^*+1}=\ldots =\mu_T$ with some $ 1<t^*<T,
$
where $t^*$ is unknown. Multivariate generalizations of these results are presented in Section B of the supplement to this article, Horv\'ath et al (2018+), but here we focus on the univariate case in order to simplify the presentation. Typical test statistics for this hypothesis testing problem are based on functionals of the CUSUM process. For example,

$$
A_T=\frac{1}{T^{1/2}}\max_{1\leq t \leq T} \left|\sum_{s=1}^tX_s-\frac{t}{T}\sum_{s=1}^TX_s\right|,
$$

which denotes the maximally selected CUSUM process, can be applied. In order to increase the power of $A_T$ when the change point $t^*$ might be close to 1 or  $T$, weighted CUSUM statistics may be used. Let $0\leq \tau<1/2$, and define the weighted version of $A_T$:
$$
A_T(\tau)=\frac{1}{T^{1/2}}\max_{1\leq t <T}\left(\frac{t}{T}\left(\frac{T-t}{T}\right)\right)^{-\tau}\left|\sum_{s=1}^tX_s-\frac{t}{T}\sum_{s=1}^TX_s\right|
$$
($A_T=A_T(0)$); we refer to Cs\"org\H{o} and Horv\'ath (1993) and Cs\"org\H{o} and Horv\'ath (1997) for a comprehensive study of such statistics. The maximally selected standardized CUSUM $A_T(1/2)$ has received special attention in the literature due to its connection with maximally selected likelihood ratio tests assuming Gaussian observations.
The statistic $A_T(1/2)$ appears in Andrews (1993), where it is shown that even when the errors in model \eqref{model-1} are independent and identically distributed, $A_T(1/2)\to \infty$ in probability under both $H_0$ and $H_A$.

Andrews (1993) introduced the statistic
$$
\bar{A}_T(1/2, t_T)=T^{-1/2}\max_{t_T\leq t \leq T-t_T}\left(\frac{t}{T}\left(\frac{T-t}{T}\right)\right)^{-1/2}\left|\sum_{s=1}^t X_s-\frac{t}{T}\sum_{s=1}^TX_s\right|,
$$

to overcome this problem, which depends on the choice of a trimming parameter $t_T$, usually taken to satisfy,
\beq\label{and-1}
t_T=\lf T\bar{\theta}\rf\quad \mbox{with some}\;\;0<\bar{\theta}<1/2.
\eeq
Under \eqref{and-1}, the convergence in distribution of $\bar{A}_T(1/2, t_T)$ follows immediately from the weak convergence of $T^{-1/2}\sum_{s=1}^{\lf Tu\rf}e_s, 0\leq u \leq 1$ when $H_0$ holds. Although $A_T(1/2)$ diverges in probability, it does have a limit distribution to an extreme value law after a suitable normalization using the results of Darling and Erd\H{o}s (1956); see Cs\"org\H{o} and Horv\'ath (1993). We note that if $t_T/T\to 0$, as $T\to \infty$, then $\bar{A}_T(1/2, t_T)$ obeys a Darling--Erd\H{o}s law.

A common feature of all of the above statistics is that they may be written as the maximal weighted difference between the sample means of the first $t$ and the last $T-t$ observations. To illustrate, simple calculation shows that
$$
\frac{1}{T^{1/2}}\left(\sum_{s=1}^tX_s-\frac{t}{T}\sum_{s=1}^TX_s\right) = \frac{t(T-t)}{T^{3/2}} \left( \frac{1}{t} \sum_{s=1}^{t} X_s - \frac{1}{T-t} \sum_{s=t+1 }^{T}X_s \right),
$$
and so the CUSUM used to define $A_T$ is the difference between the sample means of the first $t$ and the last $T-t$ observations that is down-weighted at the end points by a factor of $T^{-1/2}$. In order that change points at the beginning or end of the sample might be more apparent, one might instead consider directly the maximal difference between the sample means before and after each potential change point $t$. The statistic based on maximizing this difference over all $t$ will evidently not have a limiting distribution without down-weighting or trimming the difference near the endpoints.

We consider here the properties of
$$
D_T=D_T(t_T)=\max_{t_T\leq t \leq T-t_T}\left| \frac{1}{t}\sum_{s=1}^tX_s-\frac{1}{T-t}\sum_{s=t+1}^TX_s\right|,
$$

which we refer to as a R\'enyi statistic. The limit distribution of $D_T$ will evidently depend on the choice of the trimming parameter $t_T$. For example, if \eqref{and-1} holds, then $t_T^{1/2}D_T$ would converge in distribution to the maximum of a Gaussian process on a trimmed domain, and yield a similar test as in Andrews (1993). To increase the power of $D_T$ when a change in the mean might be near the endpoints of the sample, we assume instead that


\begin{assumption}\label{no-and}
$t_T\to\infty\;\;\mbox{\rm{and}}\;\; t_T/T\to 0,\;\;\mbox{as}\;\;T\to \infty.$
\end{assumption}
In order to establish the limit distribution of $t_T^{1/2}D_T$ under Assumption \ref{no-and}, we require a rate in the weak convergence of the partial sum process of the $e_s$'s, which we quantify with the following assumption.
\begin{assumption}\label{appr}  {\rm There are independent Wiener processes $\{W_{T,1}(x), 0\leq x \leq T/2\}$ and $\{W_{T,2}(x), 0\leq x\leq T/2\}$ such that
\beq\label{n-ap-1}
\max_{1\leq x\leq T/2}x^{-\kappa}\left|\sum_{s=1}^{\lf x \rf}e_s-\sigma W_{T,1}(x)\right|=O_P(1)
\eeq
and
\beq\label{n-ap-2}
\max_{T/2\leq x\leq T-1}(T-x)^{-\kappa}\left|\sum_{s=\lf x \rf +1}^{T}e_s-\sigma W_{T,2}(T-x)\right|=O_P(1)
\eeq
with some $\sigma>0$ and $0<\kappa<1/2$. }
\end{assumption}

The constant $\sigma^2$ can be interpreted as the long run variance of the sequence $e_s, -\infty<s<\infty$. Assumption \ref{appr} has been established for several broad classes of stationary sequences. In the case when $e_s, -\infty<s<\infty$ is a stationary and strongly mixing sequence, Assumption \ref{appr} was shown to hold in Kuelbs and Philipp (1980) (see also Bradley (2007) and Davis et al.\ (1995)). For stationary Bernoulli shift sequences under mild weak dependence conditions, which covers most popular stationary time series models, including the ARMA and GARCH sequences, this result was shown in Ling (2007) and Aue et al.\ (2014), and in this case Berkes et al.\ (2014) provide optimal rates in \eqref{n-ap-1}. Hall and Heyde (1980) obtains \eqref{n-ap-1} under various conditions for martingale difference sequences.

The limit distribution of $t_T^{1/2}D_T$ will be expressed in terms of the random variable
\beq\label{xidef}
\xi=\sup_{0\leq u \leq 1}|W(u)|,
\eeq
where $\{W(u), 0\leq u \leq 1\}$ denotes a Wiener process.

\begin{theorem}\label{main-1} If $H_0$ and Assumptions  \ref{no-and} and \ref{appr} hold, then, as $T\to \infty$, we have
$$
t_T^{1/2}\frac{D_T}{\sigma}\;\;\stackrel{{\mathcal D}}{\to}\;\;\max(\xi_1, \xi_2),
$$
where $\xi_1, \xi_2$ are independent and each have the same distribution as $\xi$ defined in \eqref{xidef}.
\end{theorem}

The proof of Theorem \ref{main-1} is provided in the online supplement to this article,  Horv\'ath et al. (2018+). This result suggests an asymptotic $\alpha$ sized test of $H_0$: given a consistent estimate of $\sigma^2$, $\hat{\sigma}^2$, examples of which we develop below, we reject if $t_T^{1/2}{D_T}/{\hat{\sigma}}$ exceeds the $1-\alpha$ quantile of  $\max(\xi_1, \xi_2)$. Below, we compare this test based on $D_T$ to other typical change point statistics, but before doing so we develop R\'enyi type statistics in more general change point problems formulated for regression models. We note that one may also obtain similar results with $D_T$ defined with asymmetric trimming, i.e. to define

$$
D_T^*=\max_{t_T\leq t \leq T-s_T}\left| \frac{1}{t}\sum_{s=1}^tX_s-\frac{1}{T-t}\sum_{s=t+1}^TX_s\right|,
$$
with $s_T$ also satisfying \ref{and-1}. We develop this case in the supplement to this paper, Horv\'ath et al. (2018+).

\section{Extension to change point tests in regression models}\label{ext}

 We consider in this section the details of applying R\'enyi statistics in two important cases: simple linear regression and nonlinear regression using both the least squares principle and generalized method of moments.

\subsection{Linear regression} Consider the simple linear regression
$$
X_t=\bx_t^\T\bbe_t+e_t,\;\;1\leq t \leq T,
$$
where $\bx_t\in \mathbb{R}^d$ and $\bbe_t\in \mathbb{R}^d$. We wish to test the null hypothesis
$$
H_0^{(1)}:\;\;\bbe_1=\bbe_2=\ldots =\bbe_T
$$
against the alternative
$$
H_A^{(1)}:\;\;\mbox{there is}\;1<t^*<T\;\mbox{such that}\;\bbe_1=\ldots =\bbe_{t^*}\neq \bbe_{t^*+1}=\ldots =\bbe_T.
$$
\begin{assumption} \label{ass-xe}
(i) The sequences $\{ \bx_t \}$ and $\{ e_t \}$ are independent.\\
(ii) The sequence  $\{ \bx_t \}$ is stationary, ergodic,  and $|\mbox{cov}(\bx_0(i), \bx_t(i))|\leq c(t+1)^{-\zeta}$ with some $c>0$ and $\zeta>2$, where $\bx_j(k)$ denotes the $k$th coordinate of $\bx_j$.\\
 (iii) The sequence $\{ e_t\}$ satisfies $Ee_0=0$ and $|Ee_te_s|\leq c(|t-s|+1)^{-\zeta}$ with some $c>0$ and $\zeta>2$.
\end{assumption}

We use the  least squares estimator $\hat\bbe_T=(\bZ_T^\top \bZ_T)^{-1}\bZ_T^\top \bX_T$, where
$$
\bZ_T =[\bx_1,...,\bx_T]^\top \in \mathbb{R}^{T\times d},
$$
and $\bX_T=(X_1, X_2,\ldots, X_T)^\T$. It follows from the ergodic theorem that
$$
\frac{1}{T}\bZ_T^\top \bZ_T \to\bC_z\;\;\mbox{a.s.}
$$

\begin{assumption}\label{ass-0}\;$\bC_z$ is a nonsingular matrix.
\end{assumption}
The residuals are defined by
\beq\label{res-lin}
\hat{e}_t=X_t-\bx_t^\T\hat{\bbe}_T, \;\;\;1\leq t \leq T.
\eeq
\begin{theorem}\label{main-lin} If $H_0^{(1)}$,  Assumptions \ref{no-and}--\ref{appr}, \ref{ass-xe} and  \ref{ass-0} hold, then, as $T\to \infty$, we have
$$
\frac{t_T^{1/2}}{\sigma}\max_{t_T\leq t \leq T- t_T}\left|\frac{1}{t}\sum_{s=1}^t\hat{e}_s-\frac{1}{T-t}\sum_{s=t+1}^T\hat{e}_s    \right|\;\;\stackrel{{\mathcal D}}{\to}\;\;\max(\xi_1, \xi_2),
$$
where  $\xi_1, \xi_2$ are independent and they have the same distribution as $\xi$ of \eqref{xidef}.
\end{theorem}

The statistic in Theorem \ref{main-lin} depends on the unknown long run variance $\sigma^2$ which can be estimated from the residuals $\hat{e}_t, 1\leq t \leq T$, as discussed in Section \ref{simu} below.

\subsection{Nonlinear regression}
We next consider the nonlinear regression model
$
X_t = h(\bx_t, \btheta_t) + e_t,  1 \le t \le T,
$
where the $\btheta_t$'s are $d$--dimensional parameter vectors.
 Under the null hypothesis
 $
 H_0^{(2)}:\btheta_1=\btheta_2=\ldots =\btheta_T
 $
 and the alternative is formulated as
$
H_A^{(2)}:$ there is $1<t^*<T$   such that $\btheta_1=\btheta_2=\ldots =\btheta_{t^*}\neq \btheta_{t^*+1}=\ldots =\btheta_T.
$
 The unknown
common parameter vector under $H_0^{(2)}$ is denoted by $\btheta_0$. Using the least squares principle, the estimator for $\btheta_0$ is the location of the minimizer of
\[
L_T(\btheta) = \sum_{t=1}^T (X_t - h(\bx_t, \btheta))^2,
\]
where the minimum is taken over the parameter space $\bTheta$. We make the following  mild assumptions:
\begin{assumption} \label{a:theta} The parameter space $\bTheta$ is
a compact subset of $\mathbb{R}^d$,  and $\btheta_0$ is an interior point of  $\bTheta$.
\end{assumption}

\begin{assumption} \label{a:h}
\begin{align*}
&\sup_{\btheta \in \bTheta} E h^2(\bx_0, \btheta) < \infty, \;\;\sup_{\btheta \in \bTheta} \Biggl \| \frac{\partial^2}{\partial \btheta^2}
h(\bx_t, \btheta) \Biggl \| \le M(\bx_t), \ \ \ EM(\bx_0) < \infty, \\
&E \Biggl \|\frac{\partial}{\partial \btheta} h(\bx_0, \btheta_0) \Biggl \|^2< \infty, \mbox{ and } E [ h(\bx_0, \btheta_0) - h(\bx_0, \btheta) ]^2 > 0, \ \ \
{\rm if} \ \ \btheta \neq \btheta_0.
\end{align*}

\end{assumption}

The conditions formulated in Assumption~\ref{a:h} ensure that
under $H_0$ the differences between the functionals based on the
unobservable errors $e_t$ and the residuals
\beq\label{res-2}
\tilde{e}_t = X_t - h(\bx_t, \hat\btheta)
\eeq
are asymptotically negligible in the sense that they do not affect the limit in Theorem  \ref{main-1} when $X_t$ is replaced by $\tilde{e}_t$.
\begin{theorem}\label{main-nonlin} If $H_0^{(2)}$,  Assumptions \ref{no-and}--\ref{appr}, \ref{ass-xe}, \ref{a:theta} and  \ref{a:h} hold, then, as $T\to \infty$, we have
$$
\frac{t_T^{1/2}}{\sigma}\max_{t_T\leq t \leq T- t_T}\left|\frac{1}{t}\sum_{s=1}^t\tilde{e}_s-\frac{1}{T-t}\sum_{s=t+1}^T\tilde{e}_s    \right|\;\;\stackrel{{\mathcal D}}{\to}\;\;\max(\xi_1, \xi_2),
$$
where  $\xi_1, \xi_2$ are independent and they have the same distribution as $\xi$ of \eqref{xidef}.
\end{theorem}


\subsection{Generalized method of moments estimation }
Our results may also be considered in the context of generalized method of moments (GMM) estimation. The basic notation that we use here is inspired by Chapter 21 of Zivot and Wang (2006). The GMM estimator $\hat{\theta}_T$ is generally the solution of a moment equation satisfying

$$
\sum_{t=1}^{T}g({\bm x}_t, \hat{\theta}_T) = 0,
$$
where ${\bm x}_t$ contains both model and instrumental variables. Let $m_t(\theta)=g({\bm x}_t, \theta)$, and we assume that the parameter $\theta \in \Theta$, where $\Theta$ is a compact subset of $\mathbb{R}$. One could more generally consider $\theta \in \mathbb{R}^d$, $d\ge 1$, which we address in the the supplement to this paper,  Horv\'ath et al. (2018+). Model stability in this case can be described as $H_0^{(3)} \; 0=Em_1(\theta_0) = \cdots =  Em_T(\theta_0), \;\; \mbox{ for a unique } \theta_0 \in \Theta,$ and a single change point at time $t^*$ is characterized by $H_A^{(3)}\; \; 0 = Em_1(\theta) = \cdots = Em_{t^*}(\theta) \ne Em_{t^*+1}(\theta) = \cdots = Em_{T}(\theta)$ for some $\theta \in \Theta$. Under $H_0$, $\theta_0$ denotes the true value of the parameter. We require that Assumption 3.1 holds for the instrumental variables, which yields that $m_t(\theta)$ is a stationary sequence. The following assumptions are standard in GMM estimation, see for example the conditions of Theorem 3.1 of Hansen (1982).

\begin{assumption}\label{gm-as-1}
$Em_0(\theta)=0$ if and only if $\theta=\theta_0$.
\end{assumption}

\begin{assumption}\label{gm-as-2}
$\sup_{\theta \in \Theta} E|m_0(\theta)| < \infty$, $\sup_{\theta \in \Theta}  \left|\frac{\partial}{\partial \theta} m_0(\theta) \right| \le M_t$, with $EM_t < \infty$, and $E\frac{\partial}{\partial \theta} m_0(\theta)$ is different from zero in a neighborhood of $\theta_0$.
\end{assumption}

\begin{assumption}\label{gm-as-3}  {\rm There are independent Wiener processes $\{W_{T,1}(x), 0\leq x \leq T/2\}$ and $\{W_{T,2}(x), 0\leq x\leq T/2\}$ such that
\beq\label{n-ap-1}
\max_{1\leq x\leq T/2}x^{-\kappa}\left|\sum_{s=1}^{\lf x \rf}m_s(\theta_0)-\sigma W_{T,1}(x)\right|=O_P(1)
\eeq
and
\beq\label{n-ap-2}
\max_{T/2\leq x\leq T-1}(T-x)^{-\kappa}\left|\sum_{s=\lf x \rf +1}^{T}m_s(\theta_0)-\sigma W_{T,2}(T-x)\right|=O_P(1)
\eeq
with some $\sigma>0$ and $0<\kappa<1/2$. }
\end{assumption}

Ling (2007) establishes general conditions under which Assumption \ref{gm-as-3} holds.

\begin{theorem}\label{gmm-thm} Suppose that $H_0^{(3)}$, Assumptions \ref{no-and},  and \ref{gm-as-1}--\ref{gm-as-3} hold, and that ${\bm x}_t$ satisfies Assumption \ref{ass-xe}(ii). Then as $T \to \infty$, we have

$$
\frac{t_T^{1/2}}{\sigma} \max_{t_T \le t \le T- t_T} \left| \frac{1}{t} \sum_{s=1}^{t}m_s(\hat{\theta}_T) - \frac{1}{T-t} \sum_{s=t+1}^{T}m_s(\hat{\theta}_T) \right| \stackrel{D}{\to} \max\{\xi_1,\xi_2\},
$$
where $\sigma^2$ is defined in Assumption \ref{gm-as-3}, and $\xi_1$ and $\xi_2$ are independent with the same distribution as $\xi$ of \eqref{xidef}.
\end{theorem}


\section{A comparison of the R\'enyi and maximally selected CUSUM statistics for end--of--sample changes}\label{sec-alt}

The primary advantage of the R\'enyi statistic over traditional test statistics is in its ability to detect change points that occur near the end points of the sample. We now study this advantage in detail by comparing the asymptotic properties of $D_T$ under the alternative of an end-of-sample change point to those of the maximally selected CUSUM statistic $A_T$. A similar analysis could be conducted for the other CUSUM based statistics described above. If $H_0$ and Assumption \ref{appr} hold, then it is well known that
\beq\label{cusnul}
\frac{1}{\sigma}A_T\;\;\stackrel{{\mathcal D}}{\to}\;\;\sup_{0\leq t \leq 1}|B(t)|,
\eeq
where $B(t),\; 0\le t \le 1$ denotes a standard Brownian bridge (see  Aue and Horv\'ath (2013)). Let $\Delta=\mu_{t^*+1}-\mu_{t^*}$ denote the size of change. In what follows, we allow both $\Delta$ and $t^*$ to depend on $T$. The necessary and sufficient conditions for the asymptotic consistency of both $A_T$ and $D_T$ are as follows:

\begin{condition}\label{ha1}
$\quad
{T^{-3/2}}{t^*(T-t^*)}|\Delta|\to\infty.
$
\end{condition}
\begin{condition}\label{ha-2}\quad
$
\min(t^*, T-t^*,t_T)t_T^{-1/2}|\Delta|\to \infty.
$
\end{condition}

\begin{theorem}\label{altha-1} (i) If Assumption \ref{appr} is satisfied, then we have that
\beq\label{alt-1}
A_T\stackrel{P}{\to}\infty.
\eeq
if and only if Condition \ref{ha1} holds. \\
(ii) If Assumptions \ref{no-and} and \ref{appr}  are  satisfied, then we have that
\beq\label{alt-2}
t_T^{1/2}D_T\stackrel{P}{\to}\infty
\eeq
if and only if Condition \ref{ha-2} holds.
\end{theorem}

\begin{rem}\label{rem-2} {\rm Consider the case  when $|\Delta|$ is constant. Writing $t^*=\lf T\theta_T\rf$, an early change point can be characterized by the assumption that $\theta_T\to 0$ (a late change, $\theta_T\to 1$, can be investigated similarly). In this case,  Condition \ref{ha1} reduces to $T^{1/2}\theta_T\to \infty$. This clearly does not hold if $\theta_T=O(T^{-1/2})$.  However, choosing $t_T$ satisfying  $t_T=O(T^{1/2})$, Condition \ref{ha-2} will be satisfied, and such a change is asymptotically detected by $D_T$. }
\end{rem}

\begin{rem}\label{rem-1} {\rm According to calculations arising in the proof of Theorem \ref{altha-1}, one has that
$$
A_T\approx {T^{-1/2}}{t^*}|\Delta|\quad\mbox{and}\quad  t_T^{1/2}D_T \approx t_T^{1/2}|\Delta|,\;\;\mbox{if}\;\;t_T\leq t^*\leq T-t_T.
$$

Therefore, the power of $D_T$ asymptotically exceeds that of $A_T$ if $T^{1/2}t_T^{1/2}/t^*\to \infty$. }
\end{rem}

Next we provide more detailed information on the local power functions of $A_T$ and $D_T$ in the case of an early change.

\begin{theorem}\label{alt-dist} We assume that $t^*/T\to 0$ as $T\to \infty$, i.e. the change is early.\\
 (i)  If $\{e_s, -\infty<s<\infty\}$ is a stationary sequence satisfying Assumption \ref{appr},
\beq\label{petrov}
E\left| \sum_{s=u}^ve_s\right|^{\bar{\nu}}\leq C(v-u+1)^{\bar{\nu}/2}\quad\mbox{with some}\;\;\bar{\nu}>2\;\;\mbox{and}\;\;C\;\;\mbox{for all}\;\;1\leq u\leq v\leq T,
\eeq
and
$
|\Delta|(t^*)^{3/2}/T\to \infty,
$
then we have that
$$
\left(\frac{T}{t^*}\right)^{1/2}\left( A_T- T^{-3/2}t^*(T-t^*)|\Delta| \right) \;\;\stackrel{{\mathcal D}}{\to}\;\;\sigma{\mathcal N},
$$
where ${\mathcal N}$ denotes a standard normal random variable. \\
(ii) If Assumptions \ref{no-and} and \ref{appr}, $t^*/t_T\to \infty$, and
$
{t_T^{1/2}t^*}|\Delta|/{T}\to 0
$
hold, then we have that
$$
t_T^{1/2}(D_T-|\Delta|)\;\;\stackrel{{\mathcal D}}{\to}\;\;\sigma\sup_{0\le u \le 1}W(u),
$$
where $W(u)$ is a standard Wiener process.
\end{theorem}

Remarks \ref{rem-2}, \ref{rem-1}, and Theorem \ref{alt-dist} suggest that $t_T$ should be chosen to be fairly small in order to detect early changes in contrast to the standard CUSUM based procedures. In the simulations and applications below we take $t_T = \lfloor \log T \rfloor$, and also consider several other potential choices of $t_T$.

\subsection{Consistency of R\'enyi statistics in regression}

We briefly discuss the consistency of our procedure based on the R\'enyi statistic constructed from the residuals $\hat{e}_t$ and $\tilde{e}_t$. The consistency of tests developed in the context of nonlinear regression can be discussed along similar lines as those based on Theorem \ref{main-lin}, and so we focus just on this latter case.

On account of Assumption \ref{ass-xe}(iii), the ergodic theorem implies that
$$
\lim_{t\to\infty}\frac{1}{t}\sum_{s=1}^t\bx_s=\bar{\bx}_0\quad\mbox{a.s.}
$$
with some $\bar{\bx}_0\in \mathbb{R}^d$. Let $\bbe_{(1)}$ and $\bbe_{(2)}$ denote the regressors before and after the change point $t^*$, respectively. We allow that both $\bbe_{(1)}$ and $\bbe_{(2)}$ might depend on $T$ and $\|\bbe_{(1)}-\bbe_{(2)}\|\to 0$ is included under the alternative. We need some minor conditions on the size and the location of the change:
\begin{assumption}\label{ass-lin-alt}\;\;\;(i)\;$t_T\leq t^*\leq T-t_T$.\;
(ii)\;\;\;$t_T^{1/2}|\bar{\bx}_0^\T(\bbe_{(1)}-\bbe_{(2)})|\to \infty.$
\end{assumption}

\begin{theorem}\label{main-lin-alt} If $H_A^{(1)}$,  Assumptions \ref{no-and}--\ref{appr}, and \ref{ass-xe}-\ref{ass-lin-alt} hold, then, as $T\to \infty$, we have
$$
{t_T^{1/2}}\max_{t_T\leq t \leq T- t_T}\left|\frac{1}{t}\sum_{s=1}^t\hat{e}_s-\frac{1}{T-t}\sum_{s=t+1}^T\hat{e}_s    \right|\;\;\stackrel{P}{\to}\;\;\infty.
$$

\end{theorem}

\section{Implementation details and a simulation study}\label{simu}

Below we provide specifications for practically implementing the proposed tests, and study these in a simulation study. The code and data used to reproduce the simulation study and data analysis below are available at {\tt{ https://github.com/ntguardian/CPAT}}.

\subsection{Estimating the long run variance}\label{var-sec} Normalizing $D_T$ so that it has a pivotal limit distribution requires the estimation of $\sigma^2$. Therefore in practice we use the statistics
 $$
 \hat{{\mathcal G}}_{T}=t_T^{1/2}\max_{t_T\leq t\leq T-t_T}\frac{1}{\hat{\sigma}_{T,t}}\left|\frac{1}{t}\sum_{s=1}^tX_s-\frac{1}{T-t}\sum_{s=t+1}^TX_s\right|,
 $$
where $\hat{\sigma}^2_{T,t}$ are estimators for $\sigma^2$ assuming the observations might contain a mean change in order to preserve monotonic power; see Perron and Vogelsang (1992) and Vogelsang (1999).  So long as the sequence of estimators $\hat{\sigma}_{T,t}$ satisfies
\beq\label{si-1}
\max_{t_T\leq t \leq T-t_T}\left|\hat{\sigma}_{T,t}-\sigma\right|=o_P(1)\quad \mbox{under}\;\;H_0
\eeq
and
\beq\label{si-2}
\left|\hat{\sigma}_{T,t^*}-\sigma\right|=o_P(1)\quad \mbox{under}\;\;H_A,
\eeq
it follows from Theorem \ref{main-1} that under $H_0$ and the assumptions of Theorem \ref{main-1},
$$
\hat{{\mathcal G}}_{T}\;\stackrel{{\mathcal D}}{\to}\;\max(\xi_1, \xi_2),
$$
where $\xi_1$ and $\xi_2$ are independent and  distributed as $\xi$ of \eqref{xidef}. Therefore, we aim to construct estimators $\hat{\sigma}_{T,t}$ that satisfy \eqref{si-1} and \eqref{si-2}. Also, under the conditions of Theorem \ref{altha-1}(b) we have under the alternative that
$$
\hat{{\mathcal G}}_{T}\;\stackrel{P}{\to}\;\infty.
$$
In case of  observations that are presumed to be uncorrelated, we use the estimators
\beq\label{simudef-1}
\hat{\sigma}^2_{T,t}=\frac{1}{T}\left(\sum_{s=1}^t(X_s-\bar{X}_t)^2+\sum_{s=t+1}^T(X_s-\tilde{X}_{T-t})^2\right),
\eeq
where
$$
\bar{X}_t=\frac{1}{t}\sum_{s=1}^tX_s\quad\mbox{and}\quad \tilde{X}_{T-t}=\frac{1}{T-t}\sum_{s=t+1}^TX_s.
$$
\begin{theorem}\label{si-th-1} We assume that $\{e_t, -\infty<t<\infty\}$ is a stationary and ergodic sequence with $Ee_0=0$ and $Ee_0^2=\sigma^2$. If Assumption \ref{appr} holds, then
the estimator $\hat{\sigma}^2_{T,t}$ defined by \eqref{simudef-1} satisfies \eqref{si-1} and \eqref{si-2}.
\end{theorem}
When the model errors are presumed to be serially correlated, we use a kernel-smoothed estimate of the spectral density at frequency zero to estimate $\sigma^2$; see, for example, Priestley (1981). The standard assumptions on the kernel $K$ and the window (smoothing parameter) $h$ will be used here:

\begin{assumption}\label{ass-k} $K\geq 0,\;\;K(0)=1,\;\;K(u)=0,\;\mbox{if}\;\;|u|>c\;\;\mbox{with some  }c>0,$ and
$K(u)$ is Lipshitz continuous on the real line.
\end{assumption}

\begin{assumption}\label{ass-h}\;\;
$
h=h(T)\to \infty$  and  $({h}/{T})(\log T)^{1/\bar{\nu}}(\log\log T)^{1/2}\to 0,
$
where $\bar{\nu}$ is given  in \eqref{petrov}.
\end{assumption}

For any $1\leq s \leq T$  and $1\leq t <T$ we define
\begin{displaymath}
X_{s,t}=\left\{
\begin{array}{ll}
X_s-\bar{X}_t,\;\;&\mbox{if}\;\;1\leq s \leq t,\vspace{.3cm}\\
X_s-\tilde{X}_{T-t},\;\;&\mbox{if}\;\;t<s\leq T.
\end{array}
\right.
\end{displaymath}
For every fix $t$ we define  the long run variance estimator       
\begin{align}\label{simudef-2}
\hat{\sigma}^2_{T,t}=\hat{\gamma}_{0,t}+2\sum_{\ell=1}^{T-1}K\left(\frac{\ell}{h}\right)\hat{\gamma}_{\ell,t}\quad\mbox{with}\quad
\hat{\gamma}_{\ell,t}=\frac{1}{T-\ell}\sum_{s=1}^{T-\ell}X_{s,t}X_{s+\ell,t}.
\end{align}

Our estimator accounts for a possible change in the mean at an unknown time. The classical estimator for $\sigma^2$ is similarly defined, one needs to replace $\bar{X}_t$ and $\tilde{X}_{T-t}$ with $\bar{X}_T$ in definitions  of the empirical covariance of lag $\ell$. If $\{e_t, -\infty<t<\infty\}$ satisfies a Bernoulli shift assumption as considered in Aue et al. (2014), then $\hat{\sigma}^2_{T,t}$ defined by \eqref{simudef-2} satisfies \eqref{si-1} and \eqref{si-2}.

\subsection{Finite sample properties of Theorem \ref{main-1}}

We now present the results of a simulation study that aimed to investigate how the limit result in Theorem \ref{main-1} manifests under $H_0$ in finite samples for several different data generating processes (DGP's). All simulations were performed using the R statistical computing language. 
We considered two different DGP's under $H_0$: 1) $e_t$ are iid normal random variables, 
2) $e_t$ are iid taking values $+1$ and $-1$ with  probability $1/2$.  
For each of these DGP's, we generated 100,000 simulated samples of lengths $T=50,250$ and $1000$, and calculated $t_T^{1/2}{D_T}/{\sigma}$. We note that we did not estimate $\sigma^2$ in this first study, but rather used the known value for each DGP. We considered three choices of $t_T$: $t_T=\lfloor \log T \rfloor$, $t_T = \lfloor T^{1/4} \rfloor $, and $t_T = \lfloor T^{1/2} \rfloor$. We only report the results for $t_T=\lfloor \log T \rfloor$ and $t_T = \lfloor T^{1/2} \rfloor$, taking $t_T=\lfloor T^{1/4} \rfloor$ performed similarly. Density estimates based on the standard normal kernel with bandwidths computed from the data were compared to the density of the limit in Theorem \ref{main-1}.  The distribution of $\xi$ can be found using Monte Carlo simulation, or using the formulas in Cs\"org\H{o} and R\'ev\'esz (1981, page 43). We employ the latter method below.

\begin{figure}
        \centering
        \caption{Plots of the density of the limit (dash) and density estimates of the the R\'enyi statistics (solid) with sample size $T=250$,  $t_T=\log T$,  $e_t$ is standard normal (upper left), $e_t$ is Bernoulli (upper right) and $t_T=\lfloor T^{1/2}\rfloor$, $e_t$ is standard normal (lower left), $e_t$ is Bernoulli (lower right) }\label{fig-0}
\begin{minipage}{.5\textwidth}
            \centering
            \includegraphics[width=.989\linewidth]{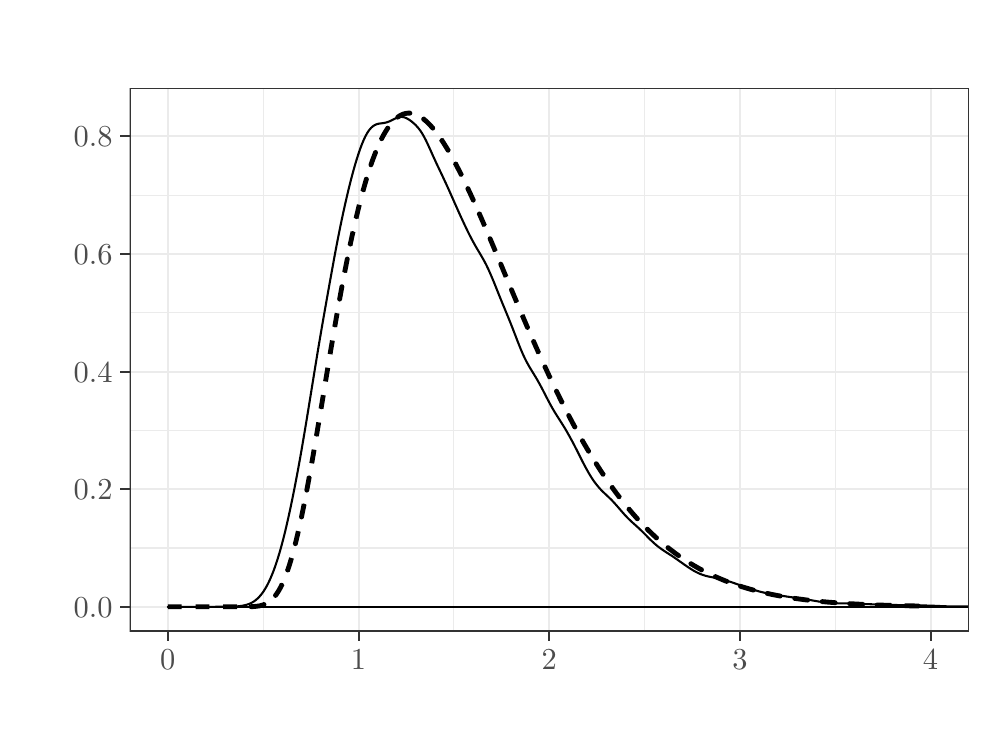}
        \end{minipage}%
        \begin{minipage}{.5\textwidth}
            \centering
            \includegraphics[width=.989\linewidth]{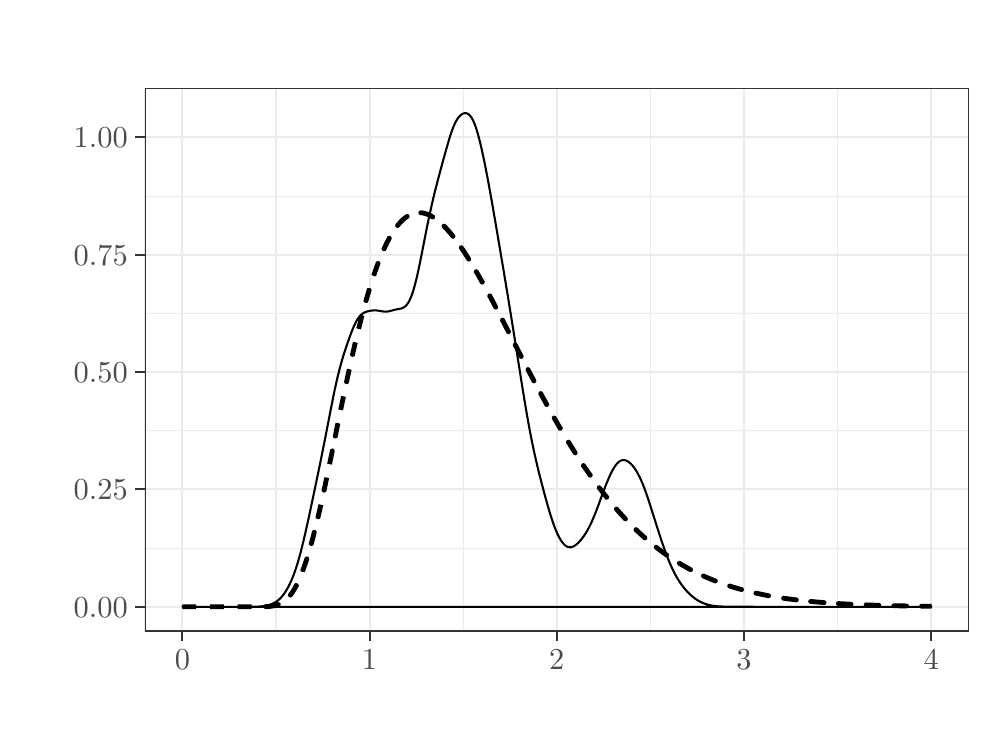}
        \end{minipage}

        \begin{minipage}{.5\textwidth}
            \centering
            \includegraphics[width=.989\linewidth]{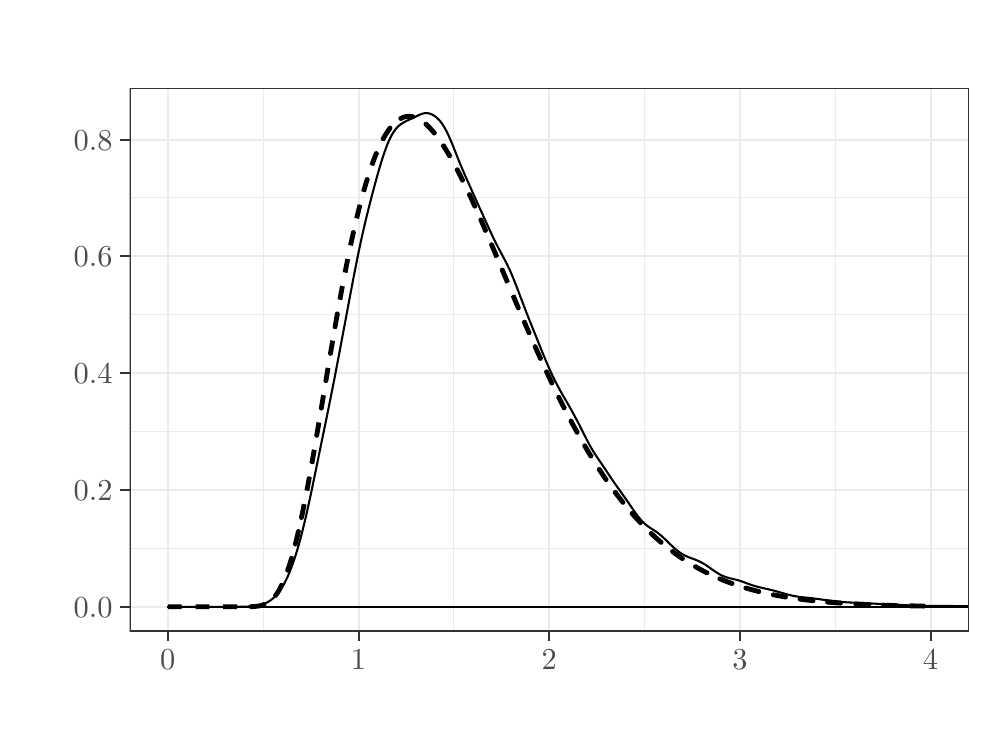}
        \end{minipage}%
        \begin{minipage}{.5\textwidth}
            \centering
            \includegraphics[width=.989\linewidth]{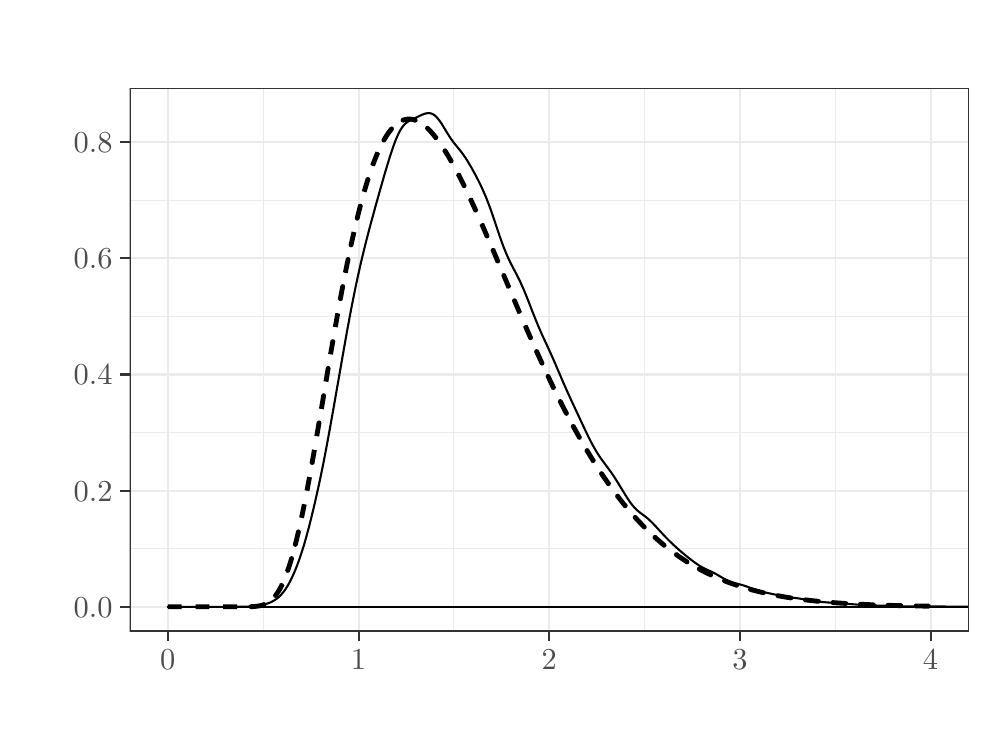}
        \end{minipage}
\end{figure}

Figure \ref{fig-0}  shows 
that the  most relevant factor determining the speed of convergence in Theorem \ref{main-1} for each DGP was the choice of $t_T$. Large choices $(t_T = \lfloor T^{1/2} \rfloor$) performed the best. But, even in the case when $t_T=\log T$, the convergence is quite fast, and importantly the difference between the sample and theoretical quantiles at the 0.9, 0.95, and 0.99 levels was small for each value of $T$ and DGP considered. \\

\subsection{Comparison of change point tests for end--of--sample changes}

We further compared the R\'enyi statistic $D_T(\lfloor \log T \rfloor)$ in the case of changes occurring early in the sample to the standard CUSUM statistic $A_T$,  the Darling-Erd\H{o}s statistic
$$E_T = (2 \log \log ({T}/({\log T})^{{3}/{2}}))^{1/2}A_T(1/2) - M_T,$$
where
$
M_T=  2 \log \log ({T}/({\log T })^{{3}/{2}})$  $ -({1}/{2})\log\log\log (T/({\log T})^{{3}/{2}}) + ({1}/{2})\log \pi,
$
and to the Lagrange multiplier statistic $\mathcal{LM}_t$ of equation (8) of Hidalgo and Seo (2013). The limiting distributions of $A_T$ and $E_T$  are reviewed  in Aue and Horv\'ath (2013). For the sake of brevity and in order to ease comparisons, we did not include the statistics of Andrews (1993) and Ling (2007), due to their relative similarity in performance to the Darling-Erd\H{o}s statistic.

Comparing to the previous subsection, we in these simulations replace the theoretical $\sigma$ with the corresponding long-run variance estimators defined in \eqref{simudef-2} with the  Bartlett kernel and bandwidth defined in Andrews (1991). For each statistic, we generated sequences $e_t$ according to  \eqref{model-1} under $H_A$ where the error sequence $e_t$ satisfied either 1) the $e_t$'s are iid standard normal, 2) $e_t$ follows the GARCH(1,1) model
$  e_t = \sigma_t w_t$ and
 $ \sigma_t^2 = 0.5 + 0.1 e_{t-1}^2 + 0.7 \sigma^2_{t-1},$ 
where $w_t$ are iid standard normal random variables, 3) the $e_t$'s follow an AR(1) model with parameter $\rho=0.5$, $e_t= 0.5 e_{t-1}+w_t$, or 4) $e_t's$ follow and ARMA(2,2) model, $e_t=0.4e_{t-1}-0.03e_{t-2}+w_t+0.5w_{t-1}-0.6w_{t-2}$. We consider the ARMA(2,2) process in order to evaluate the performance of the endogenous bandwidth selection of Andrews (1991) under model misidentification. We allowed $\Delta = \mu_T - \mu_1$ to range between $-2$ and $2$ at increments of $0.1$, and considered values of $t^* = \lfloor T^{1/4} \rfloor$ and $t^* = \lfloor 0.05T \rfloor$, which each model an early change point. We considered sample sizes of $T=50,200$ and $500$. For each simulated $X_t$ sequence parameterized by $\Delta$ and $t^*$, we calculated the statistics $D_T$, $E_T$ and $A_T$ suitably normalized by the variance estimators detailed in Section \ref{var-sec}. The greatest contrast is seen for $t^* = \lfloor T^{1/4} \rfloor$, and we present the results for only this case in order to conserve space.

For each value of $\Delta$, we calculated the percentage of test statistics that exceeded the $\alpha=0.05$ critical value of their limiting distributions. The results are reported in Figures \ref{fig-pow} and \ref{fig-small-n}. With regards to the R\'enyi statistic we report results when the trimming parameter is $t_T = \lfloor \log T \rfloor$. Results for other choices of trimming parameter for $D_T$ were similar to the logarithmic case, with trimming parameters closer to the location of the change point (for a fixed $T$) tending to perform better. We summarize these results as follows:\\
(i) In the presence of moderate serial correlation and with a small sample size, the statistics utilizing normalization by the long run variance estimator tended to be oversized, as seen in Figure \ref{fig-small-n}. This did improve with increasing sample size. (ii) We observed that the test based on $D_T(\lfloor \log T \rfloor)$ held its size well, even when using the kernel based long run variance estimator. This is especially true relative to the other CUSUM based tests, which typically had inflated size. (iii) When the change was very early, i.e. when $t^* = \lfloor T^{1/4} \rfloor$, the statistic $D_T(\lfloor \log T \rfloor)$ exhibited higher power than either the standard CUSUM, Darling-Erd\H{o}s, or Hidalgo and Seo (2013) statistics. (iv) The standard CUSUM based statistic $A_T$ displayed low empirical power for such end--of--sample changes.  Since the proportion of the observations before the change was decreasing  to zero, the power of the CUSUM and the Darling--Erd\H{o}s tests decreased or did not change as the sample size increased. The power of the R\'enyi statistic was still increasing with the sample size confirming the conclusion of Remark \ref{rem-1}. (v) Among the competing statistics, the R\'enyi statistic generally showed the best performance for end of sample change detection, followed closely by the statistic of Hidalgo and Seo (2013).

\begin{figure}
	\centering
	\mbox{		\subfigure{\includegraphics[width=2.6in]{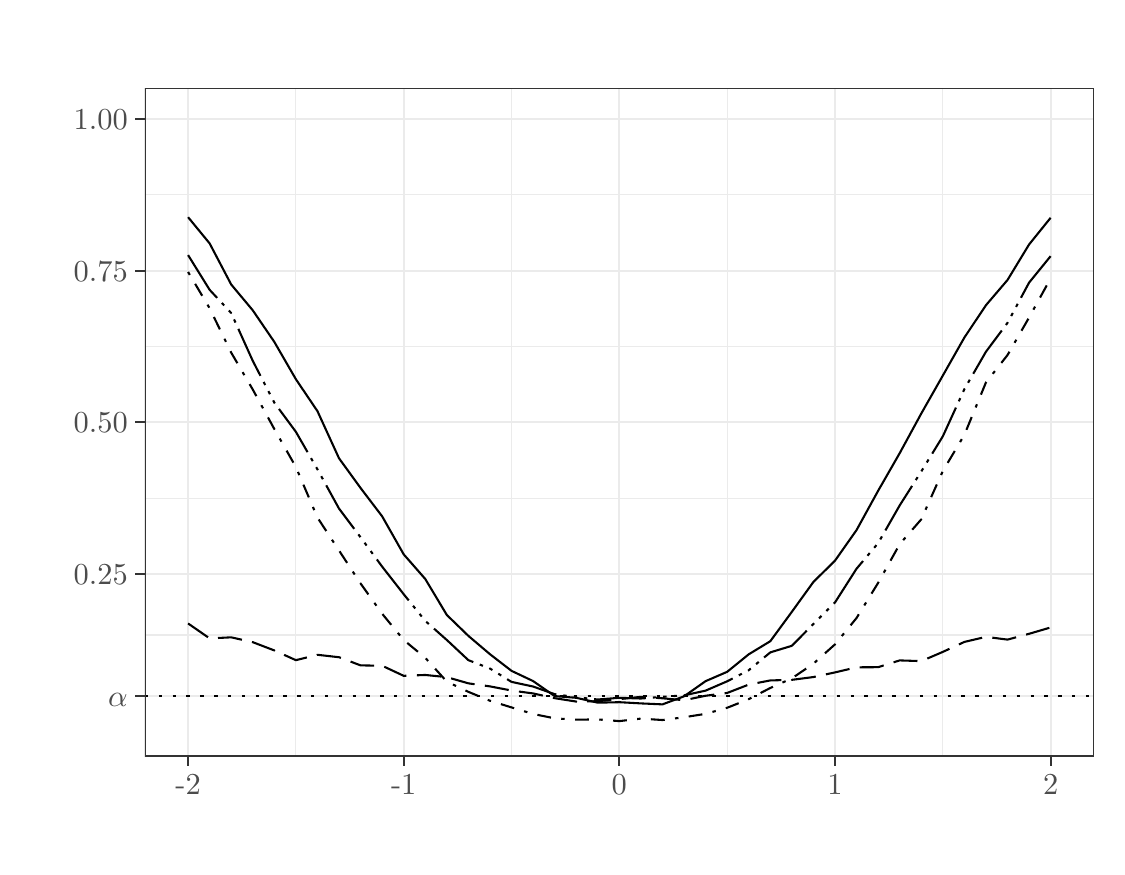} } \subfigure{\includegraphics[width=2.6in]{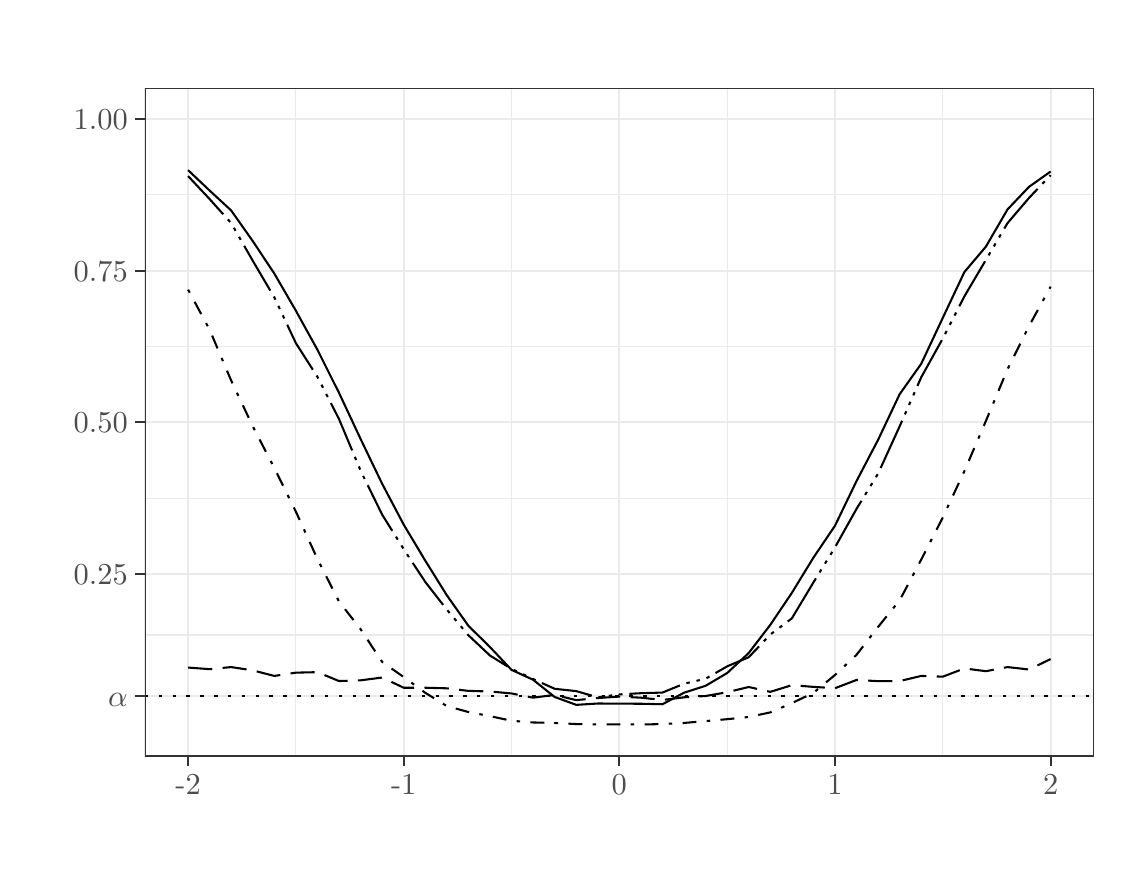} }   }   \\
	\mbox{ \subfigure{\includegraphics[width=2.6in]{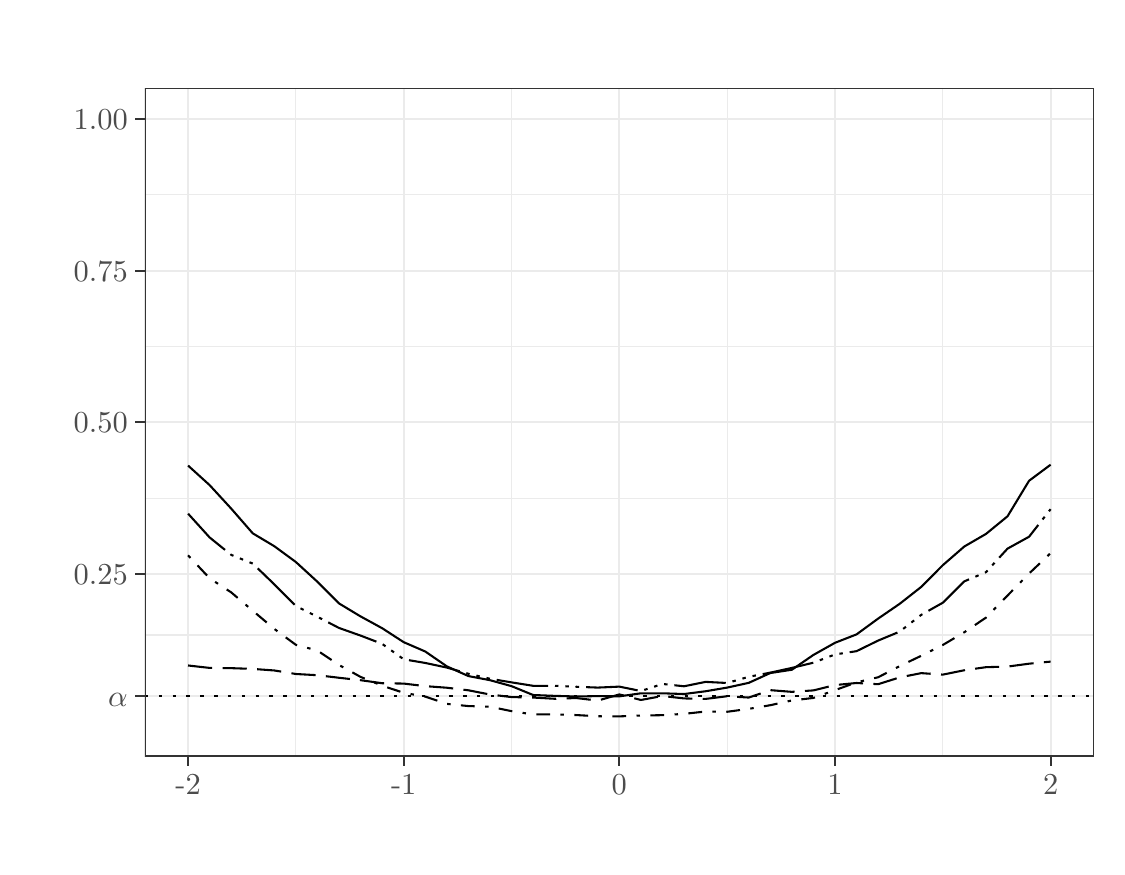} } \subfigure{\includegraphics[width=2.6in]{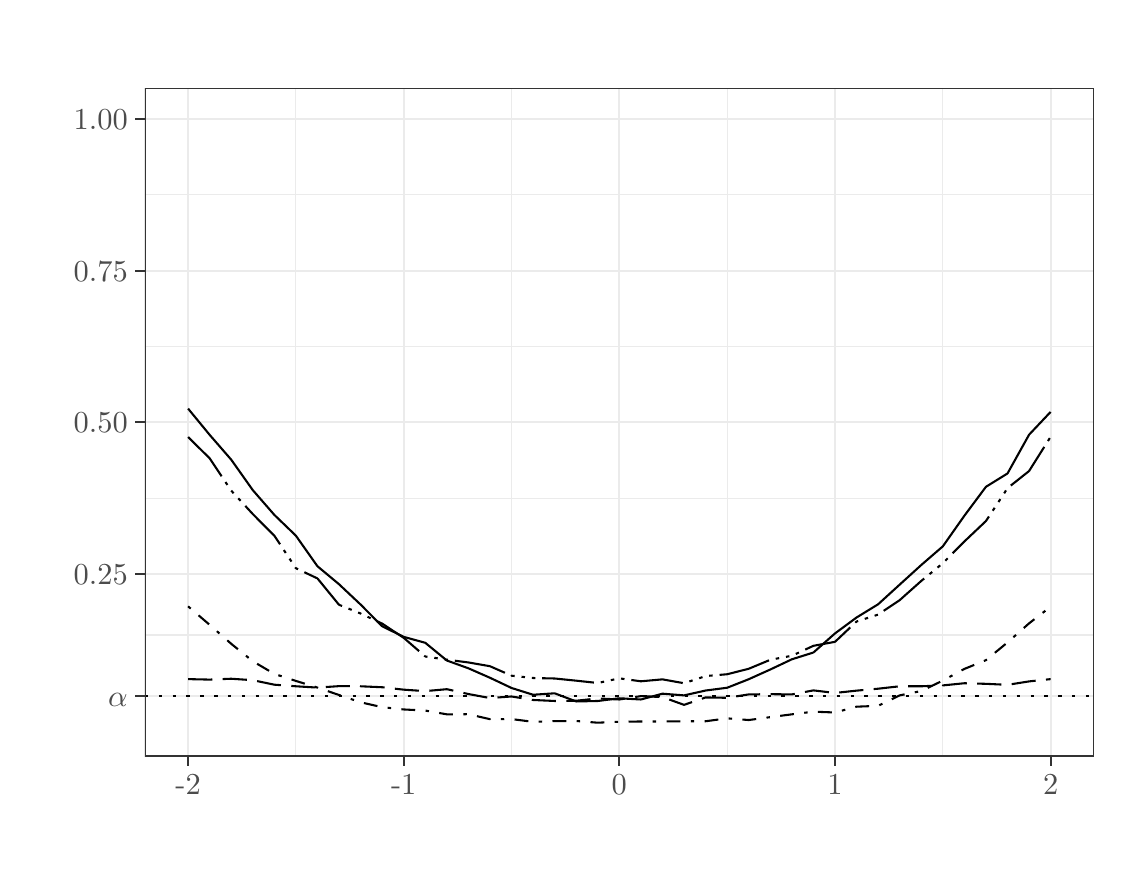} }   }  \\
	\mbox{\subfigure{\includegraphics[width=2.6in]{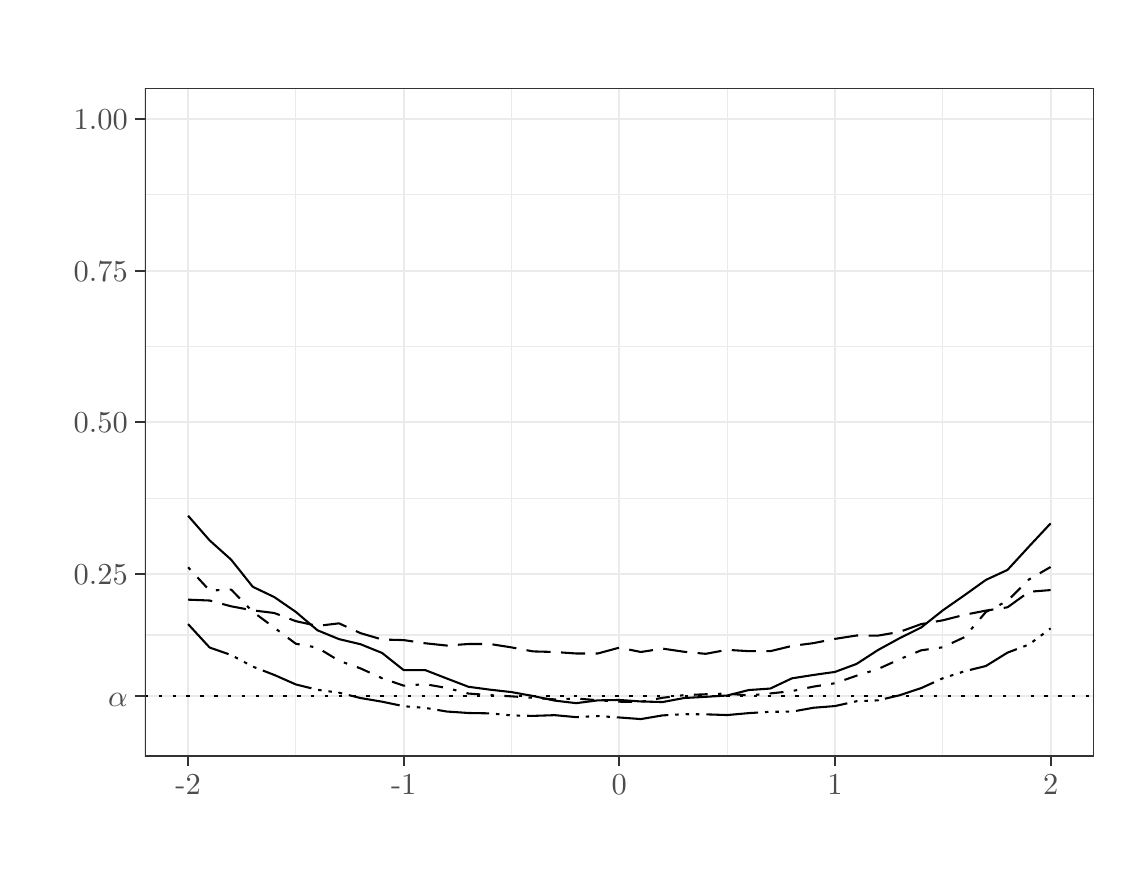} } \subfigure{\includegraphics[width=2.6in]{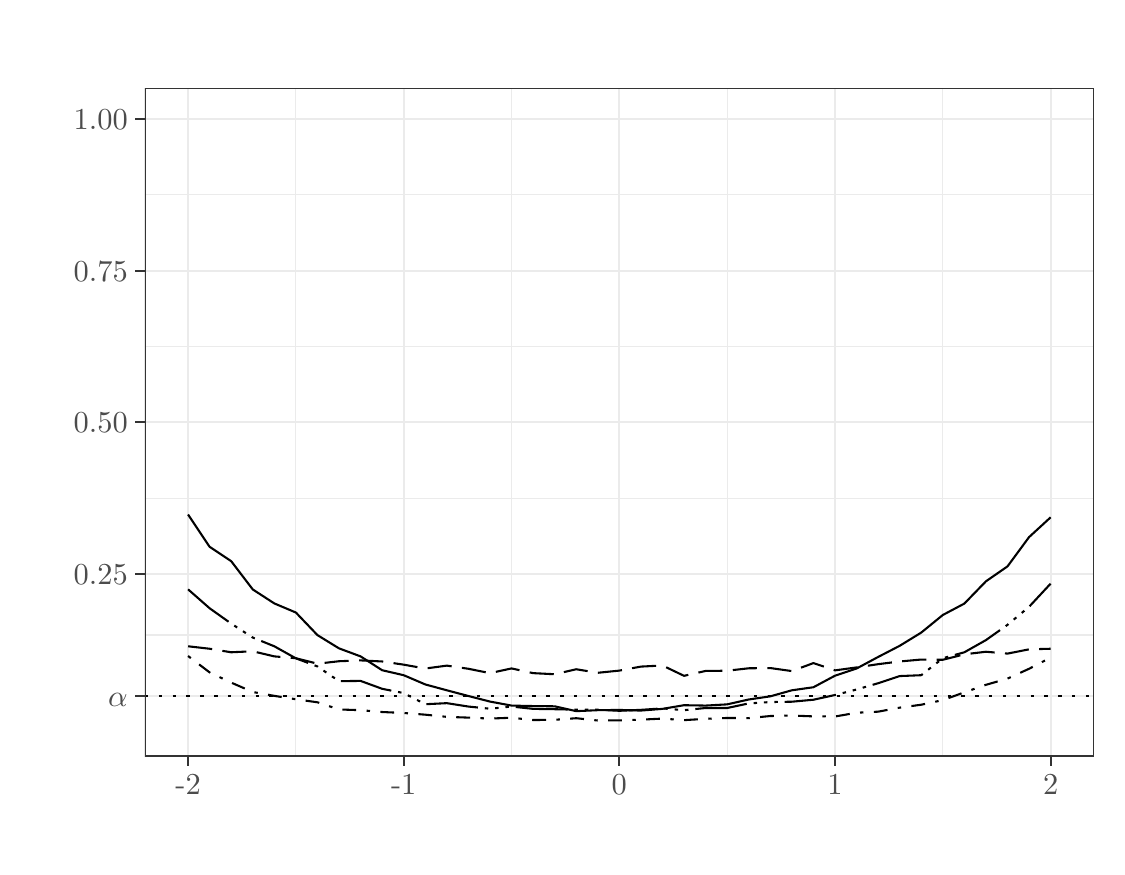} }   } \\
\mbox{\subfigure{\includegraphics[width=2.6in]{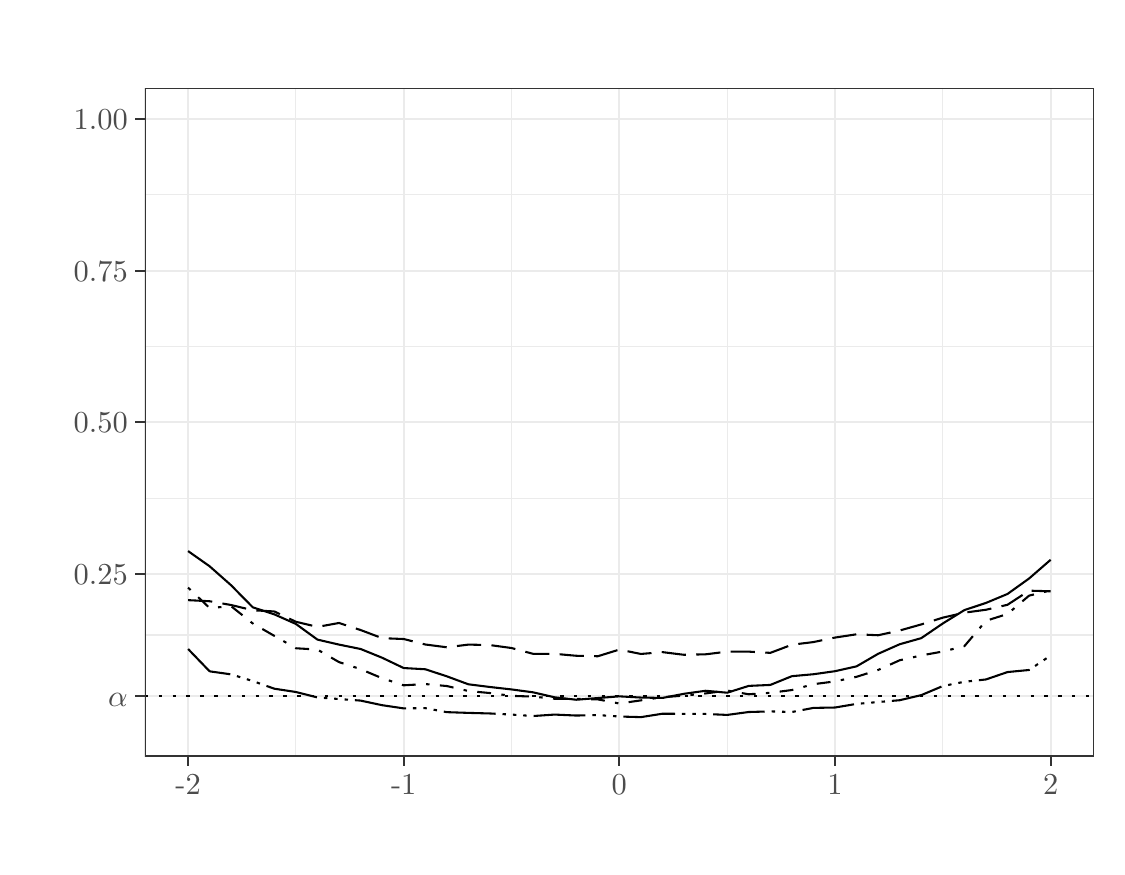}}\subfigure{\includegraphics[width=2.6in]{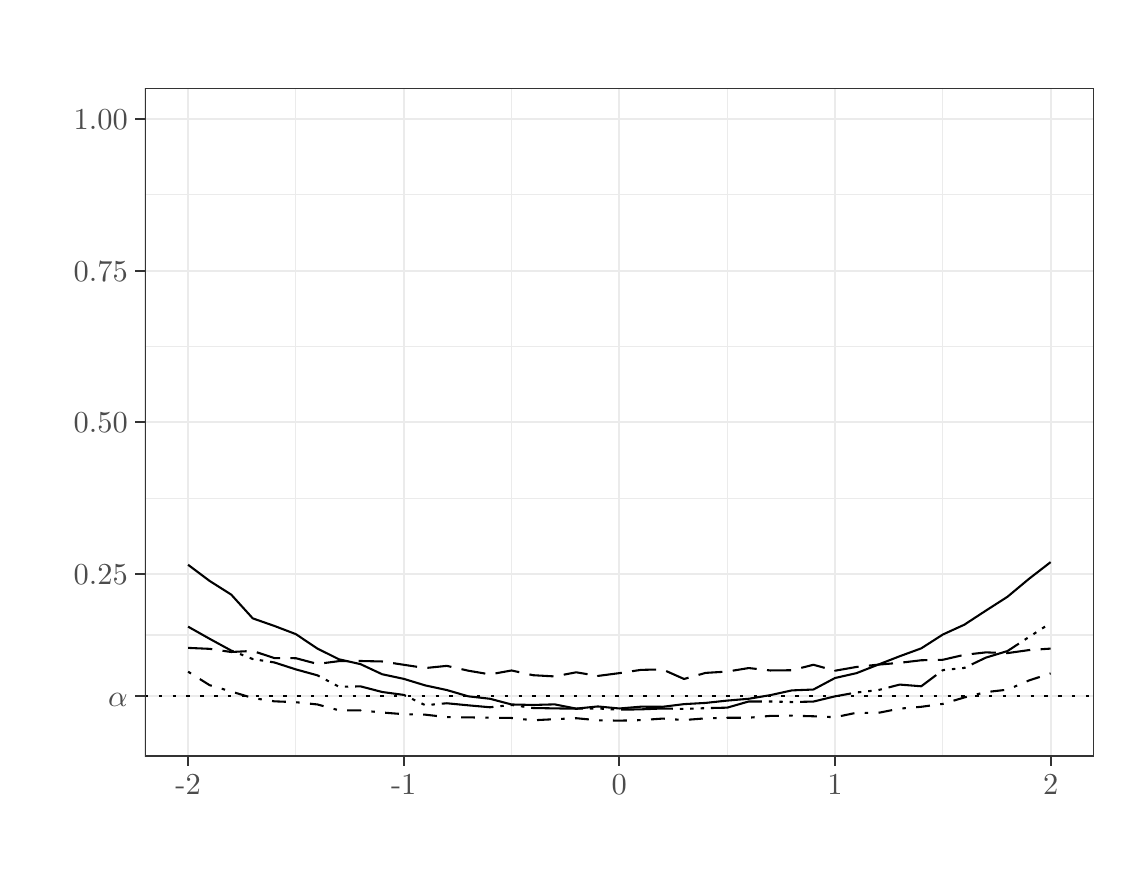} }}
	\caption{Power functions as a function of the magnitude of change with  $T = 200$ (first column) and $T = 500$ (second column) in case of iid standard normal (first row), GARCH (second row), AR(1) (third row), and ARMA(2,2) (fourth row) errors when $t^*=\lfloor T^{1/4}\rfloor$ using the R\'enyi statistic with trimming parameter $t_T=\log(T)$ (---), CUSUM, ($-\!\!$ $-$), Hidalgo--Seo ({\bf --- $\cdots$}) and Darling--Erd\H{o}s ($-$ $\cdot$ $-$) test statistics.}\label{fig-pow}

\end{figure}

\begin{figure}
  \centering
  \includegraphics[width=3.3in]{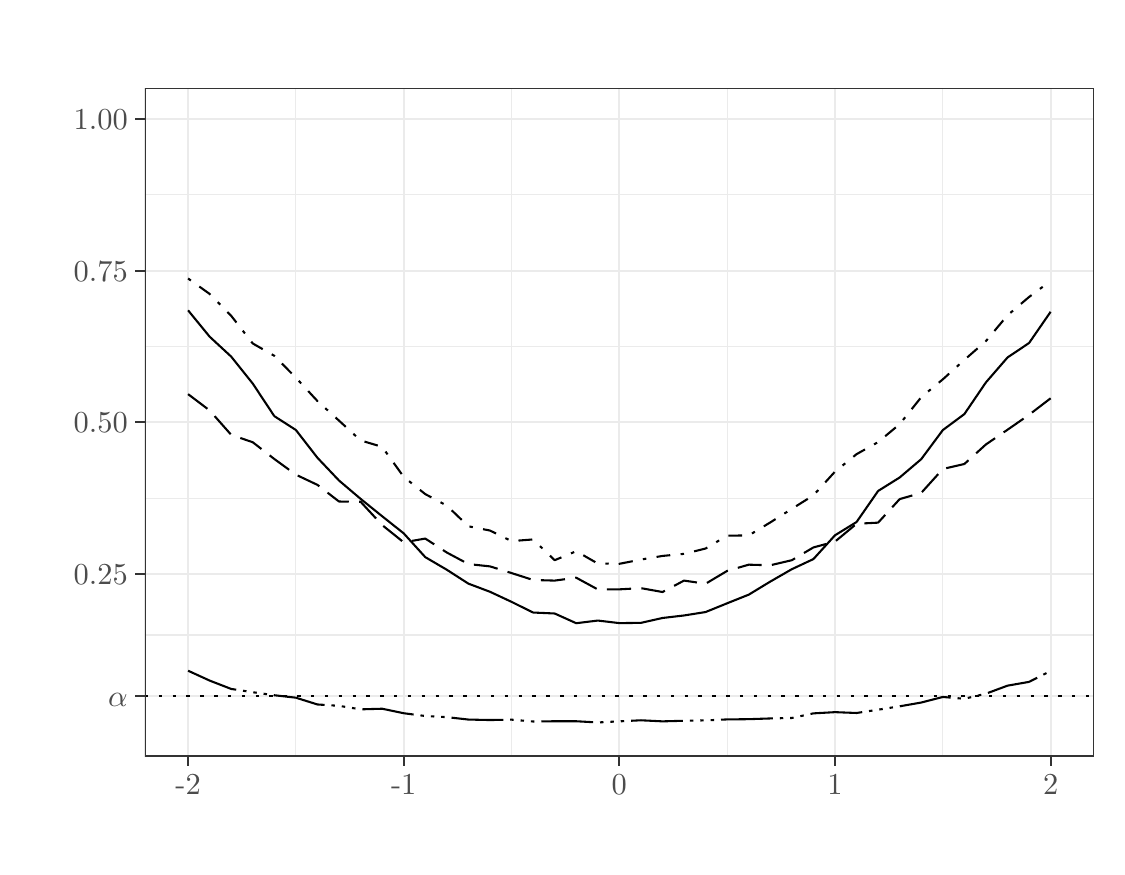}
  \caption{Power functions as a function of the magnitude of change with $T = 50$ and AR(1)  errors when $t^*=\lfloor T^{1/4}\rfloor$ using the R\'enyi statistic with trimming parameter $t_T=\log(T)$ (---), CUSUM, ($-\!\!$ $-$), Hidalgo--Seo ({\bf --- $\cdots$}) and Darling--Erd\H{o}s ($-$ $\cdot$ $-$) test statistics.}\label{fig-small-n}
\end{figure}

\section{Data Application: Application to change point testing in Fama-French factor models }\label{appl}

In this section we illustrate the R\'enyi-type test's ability to detect end-of-sample structural changes when compared to both CUSUM or the Darling-Erd\H{o}s statistic with an application to detecting structural change in the Fama-French five factor model coefficients fitted for a portfolio of banking sector stocks. In particular, we compare each statistics ability to detect a structural changes during the 2008 financial crisis.

The Fama-French five-factor model (Fama and French (2015)) is a linear regression model intended to describe the expected return of a security or portfolio of financial assets. The model takes the following form:

\begin{align}\label{ff5f}
	\begin{split}
		R_{t} - R_{Ft} &= \alpha + \beta_{M}(R_{Mt} - R_{Ft}) + \beta_{SMB} SMB_t \\
		&\;\;\;\; + \beta_{HML} HML_{t} + \beta_{RMW} RMW_t + \beta_{CMA} CMA_t + \epsilon_t.
	\end{split}
\end{align}

\noindent In \eqref{ff5f}, $R_t$ is the return of the security or portfolio at time $t$; $R_{Ft}$ is the risk-free rate of return; $R_{Mt}$ is the market return; $SMB_t$ is the return on a diversified portfolio of small stocks minus the return on a diversified portfolio of big stocks; $HML_t$ is the return of a portfolio of stocks with a high book-to-market (B/M) ratio minus the return of a portfolio of stocks with a low B/M ratio; $RMW_t$ is the returns of a portfolio of stocks with robust profitability minus a portfolio of stocks with weak profitability; and finally $CMA_t$ is the return of a portfolio of stocks with conservative investment minus the return of a portfolio of stocks with aggressive investment.

Kenneth French published data for estimating these coefficients on his website, along with portfolios he constructed. Among the portfolios are 49 industry portfolios constructed by assigning each stock on the NASDAQ, NYSE, and AMEX exchanges to an industry portfolio based on the company's four-digit SIC code. One of those portfolios represents the banking industry. Both value-weighted and equal-weighted returns are provided, and here we use the value-weighted returns. More details about the data can be found on Kenneth French's website (French (2017)).

We estimated the coefficients of \eqref{ff5f} for the banking portfolio using OLS on an expanding window of data with the starting point fixed at January 4th, 2005, and the end point varying between the trading days from August 1st to October 31st, 2008. The sample sizes then ranged from 901 days in the initial sample to 965 days for the final sample. Notably this period includes September 15th of that year on which day Lehman Brothers filed for bankruptcy. For each of these datasets we estimated the coefficients of \eqref{ff5f} and performed the CUSUM, Darling-Erd\H{o}s, and R\'enyi-type tests (with trimming parameter $t_T = \lfloor \log(T) \rfloor$ for the latter) on the residuals. We did this both with and without using the HAC variance estimation described in \eqref{simudef-2} with the Bartlett kernel and bandwidth selected using the method of Andrews (1991). We report results for the former case, since the difference was negligible. We used R and C++ for computations, including functionality provided by the \textbf{Rcpp} and \textbf{cointReg} packages; see Eddelbuettel (2013) and Aschersleben and Wagner (2016).

A visual inspection of the residuals (not shown) suggests they are heteroskedastic. We performed Ljung-Box and Box-Pierce tests on the squared residuals to test for this type of dependence for one lag, with both tests rejecting the null hypothesis of independence ($p < 0.001$).

\begin{figure}[t]\label{fig:bankexample}
	\centering
	\includegraphics[width=0.9\textwidth]{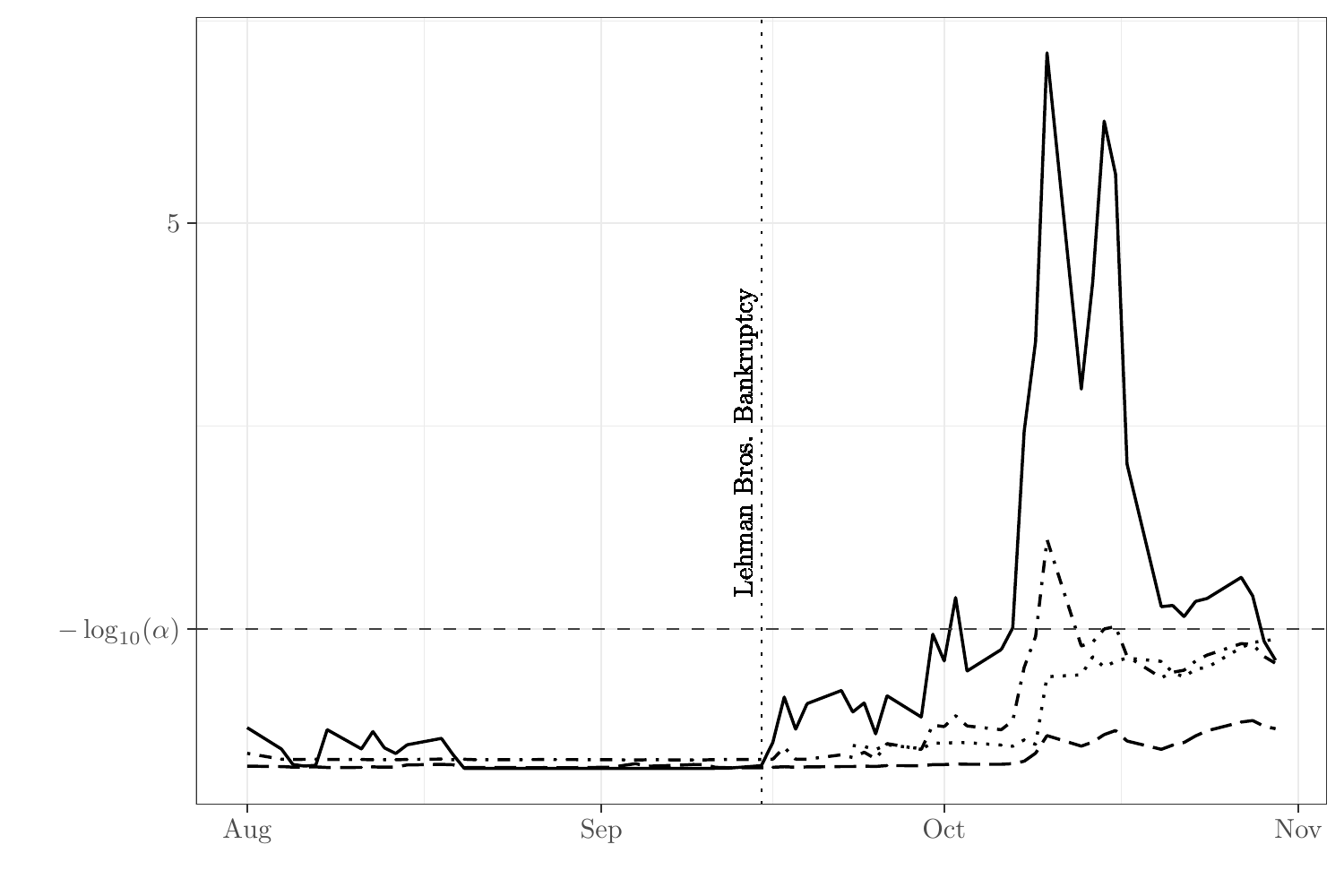}
	\caption{For $p$-value computed on a data set with end date indexed by the $x$-axis, we plot $-\log_{10}(p)$, which can be interpreted as the number of significant digits of that $p$-value. The solid line ({\bf ---}) is the R\'enyi-type statistic's $p$-value, the dotted line ($\cdots$) the CUSUM statistic's $p$-value, the dot-dash line ($-\cdot-$) the Darling-Erd\H{o}s statistic's $p$-value, and the long dashed line ({\bf ---  ---}) the $p$-values of the end-of-sample test of Andrews (2003). The horizontal line is at  $-\log_{10}{\alpha}$ with ($\alpha = 0.05$). Lehman Bros. bankruptcy on September 15th, 2008 is marked.}
\end{figure}

Figure \ref{fig:bankexample} shows the resulting p-value from each test indexed by the last date in the sample. We also compared these to the end-of-sample change point test of Andrews (2003), which assumes that the potential change point is specified by the user. As the potential change point location for Andrews (2003) method we use the date that Lehmann Brothers filed for bankruptcy, and as such the p-values for this test are only available for samples ending after this date.  From this figure we see that the R\'enyi-type test was able to retrospectively detect a structural change in the coefficients of \eqref{ff5f} with only half a month of data beyond September 15, 2008, which was roughly half a month before either the CUSUM or Darling-Erd\H{o}s tests detected changes. Interestingly, the method of Andrews (2003) was not able to detect a change in this period, which we believe can be attributed to the fact that their statistic is not as strongly weighted as, for example, the R\'enyi-type statistic.

\appendix
\begin{center}
{\bf Supplementary Material}   
\end{center}
\medskip 
 This supplement contains the proofs of the results in the main paper, as well as multivariate generalizations of Theorem 2.1. Generalizations to R\'enyi type statistics defined with asymmetric trimming are also developed. We provide the details of the consistency of the variance estimators defined in Section 5 of the main paper.

\section{Proofs of Main Results}

 \subsection{Proof of Theorem 2.1 }\label{sec-proof-1}
 First we note that under $H_0$
\begin{align}\label{basic-1}
 \max_{t_T\leq t\leq T-t_T}\left|\frac{1}{t}\sum_{s=1}^tX_s-\frac{1}{T-t}\sum_{s=t+1}^T X_s  \right|=\max(V_{T,1}, V_{T,2}),
 \end{align}
 where
 $$
 V_{T, 1}=\max_{t_T\leq t\leq T/2}\left|\frac{1}{t}\sum_{s=1}^te_s-\frac{1}{T-t}\sum_{s=t+1}^T e_s  \right|
 \quad
 \mbox{and}
 \quad
 V_{T, 2}=\max_{T/2< t\leq T-t_T}\left|\frac{1}{t}\sum_{s=1}^te_s-\frac{1}{T-t}\sum_{s=t+1}^T e_s  \right|.
 $$
 \begin{lemma}\label{lem-fir} If  Assumptions 2.1--2.2 hold, then we have
 \beq\label{fir-1}
 t_T^{1/2}V_{T, 1}=t^{1/2}_T\max_{t_T\leq t\leq T/2}\left|\frac{1}{t}\sum_{s=1}^te_s\right|+o_P(1)
 \eeq
 and
 \beq\label{fir-2}
 t_T^{1/2}V_{T, 2}=t_T^{1/2}\max_{T/2< t\leq T-t_T}\left|\frac{1}{T-t}\sum_{s=t+1}^T e_s  \right|+o_P(1).
 \eeq
 \end{lemma}
 \begin{proof}
 It is easy to see that
 $$
\left| V_{T,1}-\max_{t_T\leq t\leq T/2}\left|\frac{1}{t}\sum_{s=1}^te_s\right|\right|\leq \frac{2}{T}\max_{t_T\leq t\leq T/2}\left|\sum_{s=t+1}^T e_s  \right|
 $$
 and
 \begin{align*}
 \max_{t_T\leq t\leq T/2}\left|\sum_{s=t+1}^T e_s  \right|\leq \max_{t_T\leq t\leq T/2}\left|\sum_{s=1}^{t} e_s  \right|+\left|\sum_{s=1}^{T/2} e_s  \right|+ \left|\sum_{s=T/2+1}^T e_s  \right|.
 \end{align*}
 By Assumption 2.2 we have
 $$
 \max_{t_T\leq t\leq T/2}\left|\sum_{s=1}^{t} e_s  \right|\leq \max_{t_T\leq t\leq T/2}\left|\sum_{s=1}^{t} e_s -\sigma W_{T,1}(t) \right|+
 \sigma\max_{t_T\leq u\leq T/2}\left|W_{T,1}(u) \right|
 $$
 and
 $$
 \max_{1\leq t\leq T/2}\left|\sum_{s=1}^{t} e_s -\sigma W_{T,1}(t) \right|=O_P(T^{\kappa}).
$$
By the scale transformation of the Wiener process
$$
\sup_{0\leq x\leq T}|W_{T,1}(x)|\stackrel{{\mathcal D}}{=}T^{1/2}\sup_{0\leq u \leq 1}|W(u)|.
$$
Thus we get that
$$
\max_{t_T\leq t\leq T/2}\left|\sum_{s=1}^{t} e_s  \right|=O_P(T^{1/2})\;\;\mbox{and}\;\;\left|\sum_{s=1}^{T/2} e_s  \right|=O_P(T^{1/2}).
 $$
 Similarly,
 $$
 \left|\sum_{s=T/2+1}^T e_s  \right|=O_P(T^{1/2}).
$$
Thus we conclude
$$
\frac{1}{T}\max_{t_T\leq t\leq T/2}\left|\sum_{s=t+1}^T e_s  \right|=O_P(T^{-1/2})
$$
and therefore \eqref{fir-1} follows Assumption 2.1. Similar arguments can be used to prove \eqref{fir-2}.
 \end{proof}

 \begin{lemma}\label{we-le} If  Assumptions 2.1--2.2 hold, then we have
 \begin{align*}
 \left( t_T^{1/2}\max_{t_T\leq t\leq T/2}\Biggl| \frac{1}{t}\sum_{s=1}^te_s \Biggl|,\; \;\;t_T^{1/2}\max_{T/2< t \leq T-t_T}\Biggl|\frac{1}{T-t}\sum_{s=t+1}^Te_s\Biggl|\right)\;\;\stackrel{{\mathcal D}}{\to}\;\sigma \max (\xi_1, \xi_2),
 \end{align*}
 where $\xi_1$ and $\xi_2$ are defined in Theorem 2.1.
 \end{lemma}
 \begin{proof}
 It follows from Assumption 2.2 that
 \begin{align}\label{in-1}
 t_T^{1/2}\sup_{t_T\leq x\leq T/2} \frac{1}{\lf x\rf}\Biggl|\sum_{s=1}^{\lf x \rf} e_s -\sigma W_{T,1}(x)\Biggl|
 &=O_P\left( t_T^{1/2} \sup_{t_T\leq x\leq T/2} \frac{1}{\lf x\rf}x^{\kappa} \right)\\
 &=O_P(t^{\kappa-1/2}_T)
 =o_P(1)\notag
\end{align}
 by Assumption 2.1 and similarly
 \beq\label{in-2}
 t_T^{1/2}\sup_{T/2< x \leq T-t_T}\frac{1}{T-\lf x\rf}\Biggl|\sum_{s=\lf x\rf +1}^Te_s-\sigma W_{T,2}(T-x)\Biggl|=o_P(1).
 \eeq

 Since  the Wiener processes $W_{T,1}$ and $W_{T,2}$ are independent for all $T$, the asymptotic independence in Lemma \ref{we-le} follows from \eqref{in-1} and \eqref{in-2}. By symmetry, we need to show only that
 $$
 t_T^{1/2}\sup_{t_T\leq x\leq T/2} \frac{1}{\lf x\rf}\Biggl| W_{T,1}(x)\Biggl|\;\;\stackrel{{\mathcal D}}{\to}
 \;\;\xi,
 $$
 where $\xi$ is defined above Theorem 2.1. By the scale transformation of the Wiener process
 $$
 t_T^{1/2}\sup_{t_T\leq x\leq T/2} \frac{1}{ x}\Biggl| W_{T,1}(x)\Biggl|=t_T^{1/2}\sup_{1\leq y \leq T/(2t_T)}\left|\frac{1}{yt_T}W_{T, 1}(yt_T) \right|
 \stackrel{{\mathcal D}}{=}\sup_{1\leq y \leq T/(2t_T)}\left|\frac{W(y)}{y} \right|,
 $$
 where $W$ denotes a Wiener process. Since by Assumption 2.1, $T/t_T\to \infty$, elementary arguments give that
 $$
 \sup_{1\leq y \leq T/(2t_T)}\left|\frac{W(y)}{y} \right|\to \sup_{1\leq y <\infty}\left|\frac{W(y)}{y} \right|\;\;\;\mbox{a.s.}\quad \quad\mbox{and}\quad\quad
\sup_{1\leq y <\infty}\left|\frac{W(y)}{y} \right|\stackrel{{\mathcal D}}{=}\xi,
 $$
 completing the proof of the lemma.
 \end{proof}
 \noindent
 {\it Proof of Theorem 2.1.} It follows immediately from Lemmas \ref{lem-fir} and \ref{we-le}.
 \qed



 \subsection{Proof of Theorems 3.1, 3.2, and 3.3} \label{pr-lin-non}

 \noindent
 {\it Proof of Theorem 3.1.} It follows from G\'orecki et al.\ (2016) that
 $$
 T^{1/2}\|\hat{\bbe}_T-\bbe_0\|=O_P(1),
 $$
 where $\bbe_0$ denotes the common coefficient vector under $H_0^*$. Using the definition of the residuals $\hat{e}_t$ we get
 \begin{align*}
 \frac{1}{t}&\sum_{s=1}^t\hat{e}_s-\frac{1}{T-t}\sum_{s=t+1}^T\hat{e}_s
 =\frac{1}{t}\sum_{s=1}^t{e}_s-\frac{1}{T-t}\sum_{s=t+1}^T{e}_s +
 \left(\frac{1}{t}\sum_{s=1}^t\bx_s-\frac{1}{T-t}\sum_{s=t+1}^T\bx_s \right)^\T(\bbe_0-\hat{\bbe}_T).
 \end{align*}
It follows from Assumptions 2.1, 2.2 and 3.1 that
$$
\max_{t_T\leq t\leq T-t_T}\left\|  \frac{1}{t}\sum_{s=1}^t\bx_s-\frac{1}{T-t}\sum_{s=t+1}^T\bx_s  \right\|=o_P(1),
$$
and therefore
\begin{align*}
t_T^{1/2}\max_{t_T\leq t\leq T-t_T}\left| \left(\frac{1}{t}\sum_{s=1}^t\bx_s-\frac{1}{T-t}\sum_{s=t+1}^T\bx_s \right)^\T(\bbe_0-\hat{\bbe}_T)   \right|=o_P((t_T/T)^{1/2})=o_P(1).
\end{align*}
\qed


\noindent
{\it Proof of Theorem 3.2.} G\'orecki et al.\ (2016) proved  that
 \beq\label{rate-bth}
 T^{1/2}\|\hat{\btheta}_T-\btheta_0\|=O_P(1).
 \eeq
Hence using a two term Taylor expansion with   Assumption 3.4 we conclude
\begin{align*}
&\frac{t_T^{1/2}}{\sigma}\max_{t_T\leq t \leq T- t_T}\left|\frac{1}{t}\sum_{s=1}^t\tilde{e}_s-\frac{1}{T-t}\sum_{s=t+1}^T\tilde{e}_s    \right|\\
&=\frac{t_T^{1/2}}{\sigma}\max_{t_T\leq t \leq T- t_T}\Biggl|\frac{1}{t}\sum_{s=1}^t{e}_s-\frac{1}{T-t}\sum_{s=t+1}^T{e}_s
+\Biggl( \frac{1}{t} \sum_{s=1}^t\frac{\partial}{\partial \btheta}h(\bx_s, \btheta_0)\\
&\hspace{2cm}-\frac{1}{T-t}\sum_{s=t+1}^T\frac{\partial}{\partial\btheta} h(\bx_s, \btheta_0)     \Biggl)^\T (\btheta_0- \hat{\btheta}_T) \Biggl|
+O_P\left(t_T^{1/2}\frac{1}{T}\right).
\end{align*}
Using again Assumption 3.4 with \eqref{rate-bth} we get that
\begin{align*}
{t_T^{1/2}}&\max_{t_T\leq t \leq T- t_T}\Biggl|\Biggl(\frac{1}{t}  \sum_{s=1}^t\frac{\partial}{\partial \btheta}h(\bx_s, \btheta_0)
-\frac{1}{T-t}\sum_{s=t+1}^T\frac{\partial}{\partial\btheta} h(\bx_s, \btheta_0)     \Biggl)^\T (\btheta_0- \hat{\btheta}_T) \Biggl|
=o_P(1).
\end{align*}
The result now follows from Theorem 2.1.
\qed

\noindent
 {\it Proof of Theorem 3.3.}

 It follows from Assumptions 3.1 and 3.6 that for all $\theta \in \Theta$,

 \begin{align}\label{gm-pf}
   \frac{1}{T} \sum_{t=1}^{T} m_t(\theta) \to Em_0(\theta), \;\; a.s.
 \end{align}
By the mean value theorem for all $\theta,\theta' \in \Theta$, we have that for some $\nu \in (\theta,\theta')$,
\begin{align}\label{gm-pft}
|m_t(\theta)  - m_t(\theta') | =\left|\frac{\partial}{\partial\theta}m_t(\nu)\right||\theta - \theta'| \le M_t |\theta - \theta'|,
\end{align}
where according to Assumption 3.6  $EM_t< \infty$. Putting together \eqref{gm-pf} and \eqref{gm-pft} and using the assumption that $\Theta$ is compact, we get that

$$
\sup_{\theta \in \Theta} \left| \frac{1}{T} \sum_{t=1}^{T} m_t(\theta) - Em_0(\theta) \right| \to 0,\; a.s.,
$$
which yields that

\begin{align*}
\hat{\theta}_T \to \theta_0, \;\; a.s.
\end{align*}
Again by the mean value theorem we have that

$$
\frac{1}{T}\sum_{t=1}^{T} m_t(\hat{\theta}_T) - \frac{1}{T}\sum_{t=1}^{T} m_t(\theta_0) = (\hat{\theta}_T - \theta_0) \frac{1}{T} \sum_{t=1}^{T}\frac{\partial}{\partial\theta}m_t(\nu_T),
$$
where $\nu_T$ satisfies $|\nu_T - \theta_0| \le |\hat{\theta}_T - \theta_0|$. By Assumption 3.6, we have that

 $$
 \frac{1}{T} \sum_{t=1}^{T} \frac{\partial}{\partial \theta} m_t(\hat{\theta}_T) \to E\frac{\partial}{\partial \theta} m_t(\theta_0), \; a.s.
 $$
Hence by Assumption 3.7, we get that

\begin{align}\label{gm-pf-1}
 |\hat{\theta}_T - \theta_0| = O_P(T^{-1/2}).
\end{align}
Applying Assumptions 2.1, 3.7 and \eqref{gm-pf-1}, we get that

\begin{align}\label{gm-pf-2}
t_T^{1/2}  \max_{t_T \le t \le T-t_T} \left| \frac{1}{t} \sum_{s=1}^{t}(m_t(\hat{\theta}_T)-m_t(\theta_0))\right| = o_P(1),
\end{align}

\begin{align}\label{gm-pf-3}
t_T^{1/2}  \max_{t_T \le t \le T-t_T} \left| \frac{1}{T-t} \sum_{s=t+1}^{T}(m_t(\hat{\theta}_T)-m_t(\theta_0))\right| = o_P(1).
\end{align}

Indeed by the mean value theorem, we have that

$$
\frac{1}{t} \sum_{s=1}^{t}( m_t(\hat{\theta}_T)- m_t(\theta_0)) = (\hat{\theta}_T- \theta_0 ) \frac{1}{t} \sum_{s=1}^{t} \frac{\partial}{\partial \theta} m_t(\nu_t),
$$
with $\nu_t \in (\hat{\theta}_T, \theta_0)$, and \eqref{gm-pf-1} and \eqref{gm-pf-2} imply that

$$
\max_{t_T \le t \le T-t_T} \left| \frac{1}{t} \sum_{s=1}^{t} \frac{\partial}{\partial \theta} m_t(\nu_t)\right| = O_P(1),
$$
which establishes \eqref{gm-pf-2}. \eqref{gm-pf-3} follows similarly. Now the result follows from \eqref{gm-pf-2}, \eqref{gm-pf-3}, Assumption 3.7, and Theorem 2.1.
 \subsection{Proof of Theorems 4.1--4.3}\label{sec-proof-2}
Since the change in mean occurs at time $t^*$ we have that
\beq\label{repa-1}
T^{-1/2}\left(\sum_{s=1}^tX_s-\frac{t}{T}\sum_{s=1}^TX_s\right)= T^{-1/2}\left(\sum_{s=1}^te_s-\frac{t}{T}\sum_{s=1}^Te_s\right)  -z_{t,T},
\eeq
where
\begin{displaymath}
z_{t,T}=\left\{
\begin{array}{ll}
  T^{-3/2} t(T -t^*)\Delta   &\mbox{if}\;\;1\leq t \leq t^*,
\vspace{.2cm}\\
T^{-3/2}t^*(T-t)\Delta  &\mbox{if}\;\;t^*<t\leq T
\end{array}
\right.
\end{displaymath}
and
\beq\label{repa-2}
t_T^{1/2}\left(\frac{1}{t}\sum_{s=1}^tX_s-\frac{1}{T-t}\sum_{s=t+1}^TX_s\right)=t_T^{1/2}\left(\frac{1}{t}\sum_{s=1}^te_s-\frac{1}{T-t}\sum_{s=t+1}^Te_s\right)  -v_{t,T},
\eeq
where
\begin{equation}\label{thm-fix}
v_{t,T}=\left\{
\begin{array}{ll}
\displaystyle t_T^{1/2}\frac{T-t^*}{T-t}\Delta,   &\mbox{if}\;\;1\leq t \leq t^*,
\vspace{.2cm}\\
\displaystyle t_T^{1/2}\frac{t^*}{t}\Delta,  &\mbox{if}\;\;t^*<t\leq T.
\end{array}
\right.
\end{equation}
\noindent
{\it Proof of Theorem 4.1.} It follows from \eqref{repa-1} that
\begin{align*}
\max_{1\leq t \leq T}|z_{t,T}|&-\max_{1\leq t \leq T}  T^{-1/2}\left|\sum_{s=1}^te_s-\frac{t}{T}\sum_{s=1}^Te_s \right|\\
& \leq A_T
\leq \max_{1\leq t \leq T}|z_{t,T}|+\max_{1\leq t \leq T}  T^{-1/2}\left|\sum_{s=1}^te_s+\frac{t}{T}\sum_{s=1}^Te_s \right|,
\end{align*}
and therefore (4.1) from the main text and Condition 4.1 imply (4.2) of the main paper. Similarly,
\begin{align*}
\max_{t_T\leq t \leq T-t_T}|v_{t,T}|&-\max_{t_T\leq t \leq T-t_T}t_T^{1/2}\left|\frac{1}{t}\sum_{s=1}^te_s-\frac{1}{T-t}\sum_{s=t+1}^Te_s\right|\\
&\leq t_T^{1/2}D_T\leq
\max_{t_T\leq t \leq T-t_T}|v_{t,T}|+\max_{t_T\leq t \leq T-t_T}t_T^{1/2}\left|\frac{1}{t}\sum_{s=1}^te_s-\frac{1}{T-t}\sum_{s=t+1}^Te_s\right|.
\end{align*}
Hence (4.3) of the main paper follows from Theorem 2.1 and Condition 4.2.
\qed

\begin{lemma}\label{alt-lem-1} If (4.4) from the main paper holds, then we have that
\beq\label{rad-1}
\max_{1\leq t \leq t^*}\frac{1}{(t^*-t+1)^{1/2}}\left|\sum_{s=t}^{t^*}e_s\right|=O_P((\log T)^{1/\bar{\nu}})
\eeq
and
\beq\label{rad-2}
\max_{t^*<t \leq  T}\frac{1}{(t-t^*+1)^{1/2}}\left|\sum_{s=t^*}^{t}e_s\right|=O_P((\log T)^{1/\bar{\nu}}).
\eeq
Furthermore, for all $x>0$ we have
\begin{align}\label{rad-3}
P\left\{\max_{t^*-A\leq t \leq t^*}\frac{1}{t^*-t+1}\sum_{s=t}^{t^*}e_s>xA^{-1/2}\right\}\leq \frac{c_*}{x^{\bar{\nu}}}
\end{align}
and
\begin{align}\label{rad-4}
P\left\{\max_{t^*+A\leq t \leq T}\frac{1}{t-t^*+1}\sum_{s=t^*}^{t}e_s>xA^{-1/2}\right\}\leq \frac{c_*}{x^{\bar{\nu}}}
\end{align}
with some constant $c_*$.
\end{lemma}
\begin{proof} With $\bar{e}_s=e_{t^*-s+1}, 1\leq s \leq t^*$  and $\bar{e}_s=0, s>t^*$ we have
\begin{align*}
\max_{1\leq t \leq t^*}\frac{1}{(t^*-t+1)^{1/2}}\left|\sum_{s=t}^{t^*}e_s\right|&=\max_{1\leq t \leq t^*}\frac{1}{t^{1/2}}\left|\sum_{s=1}^{t}\bar{e}_s\right|
\leq \max_{1\leq u\leq \log t^*} \max_{e^{u-1}\leq t\leq e^u}\frac{1}{t^{1/2}}\left|\sum_{s=1}^{t}\bar{e}_s\right|\\
&\leq \max_{1\leq u\leq \log t^*}e^{-(u-1)/2} \max_{e^{u-1}\leq t\leq e^u}\left|\sum_{s=1}^{t}\bar{e}_s\right|.
\end{align*}
By the the maximal inequality of Billingsley (1968) (cf.\ also  Corollary 3.1 of M\'oricz et al.\ (1982)) we have
\beq\label{pet-1}
E\left(\max_{e^{u-1}\leq t\leq e^u}\left|\sum_{s=1}^{t}\bar{e}_s\right|\right)^{\bar{\nu}}\leq c_1e^{u\bar{\nu}/2}
\eeq
with some constant $c_1$. Hence Markov's inequality yields for all $x>0$ that
\begin{align*}
P\left\{\max_{1\leq t \leq t^*}\frac{1}{(t^*-t+1)^{1/2}}\left|\sum_{s=t}^{t^*}e_s\right|>x\right\}&\leq
\frac{e^{\bar{\nu}/2}}{x^{\bar{\nu}}}\sum_{u=0}^{\log t^*}e^{-u\bar{\nu}}E\max_{e^{u-1}\leq t\leq e^u}\left(\sum_{s=1}^{t}\bar{e}_s\right)^{\bar{\nu}}
\leq \frac{c_2}{x^{\bar{\nu}}}\log T
\end{align*}
with $c_2=c_1e^{\bar{\nu}/2}$, completing the proof of \eqref{rad-1}. Similar arguments yield \eqref{rad-2}.\\
Following the proof of \eqref{rad-1} we have
\begin{align*}
\max_{1\leq t \leq t^*-A}\frac{1}{t^*-t+1}\left|\sum_{s=t}^{t^*}e_s\right|
\leq \max_{\log A\leq u\leq \log t^*}e^{-(u-1)} \max_{e^{u-1}\leq t\leq e^u}\left|\sum_{s=1}^{t}\bar{e}_s\right|.
\end{align*}
Combining \eqref{pet-1} with Markov's inequality we have
\begin{align*}
P\left\{\max_{1\leq t \leq t^*-A}\frac{1}{t^*-t+1}\left|\sum_{s=t}^{t^*}e_s\right|>x\right\}&\leq
\frac{e^{\bar{\nu}}}{x^{\bar{\nu}}}\sum_{u=\log A}^{\log t^*}e^{-u\bar{\nu}}E\max_{e^{u-1}\leq t\leq e^u}\left(\sum_{s=1}^{t}\bar{e}_s\right)^{\bar{\nu}}\\
&\leq \frac{c_1e^{\bar{\nu}}}{x^{\bar{\nu}}}\sum_{u=\log A}^\infty e^{-u\bar{\nu}/2}
\leq \frac{c_*}{x^{\bar{\nu}}A^{\bar{\nu}/2}}.
\end{align*}
Hence the proof of \eqref{rad-3} is complete and similar arguments give \eqref{rad-4}.
\end{proof}

\noindent
{\it Proof of Theorem 4.2.} We  can assume without loss of generality that  $\Delta>0$. It follows from Lemma \ref{alt-lem-1} that
\begin{align}\label{gc}
\max_{1\leq t <t^*}(t^*-t)^{-1/2}\left\{\left| \sum_{s=t}^{t^*}e_s   \right|+\frac{t^*-t}{T^{3/2}}\left|\sum_{s=1}^Te_s\right|\right\}=O_P((\log T)^{1/\bar{\nu}}).
\end{align}
We note that there is a constant $c_1>0$ such that
$
z_{t^*,T}-z_{t,T}\geq c_1(t^*-t) \Delta T^{-1/2}$,  if $1\leq t \leq t^*,
$
and therefore for all $c_2$ we have
\beq\label{gr-1}
\limsup_{T\to\infty}\max_{1\leq t\leq t^*-t_{1}}(T^{1/2}(z_{t,T}-z_{t^*,T})+c_2(t^*-t)^{1/2}(\log T)^{1/\bar{\nu}})=-\infty,
\eeq
where $t_{1}=(\log T)/\Delta^2$. Next we observe that
\begin{align}\label{noabs-1}
\lim_{T\to \infty}P\Biggl\{\max_{1\leq t\leq t^*- t_{1}}&\left|\sum_{s=1}^tX_s-\frac{t}{T}\sum_{s=1}^TX_s\right|-\left|\sum_{s=1}^{t^*}X_s-\frac{t^*}{T}\sum_{s=1}^TX_s\right|\\
&=\max_{1\leq t\leq t^*- t_{1}}\left(\sum_{s=1}^tX_s-\frac{t}{T}\sum_{s=1}^TX_s\right)-\left(\sum_{s=1}^{t^*}X_s-\frac{t^*}{T}\sum_{s=1}^TX_s\right)
\Biggl\}=1.\notag
\end{align}
Hence it follows from \eqref{gc}--\eqref{noabs-1} that $t^*\Delta^2\to \infty$ implies
$$
\limsup_{T\to \infty}P\left\{\max_{1\leq t\leq t^*- t_{1}}\left|\sum_{s=1}^tX_s-\frac{t}{T}\sum_{s=1}^TX_s\right|-\left|\sum_{s=1}^{t^*}X_s-\frac{t^*}{T}\sum_{s=1}^TX_s\right|>-x\right\}=0
$$
for all $x>0$. \\
Let $t_{C,1}=C/\Delta^2$. We note that by \eqref{rad-3} for all $C>2$
\begin{align*}
\lim_{\alpha\to \infty}P\Biggl\{\max_{t^*-t_1\leq t \leq t_{C,1}}\frac{1}{t^*-t+1}\left|\sum_{s=t}^{t^*}e_s\right|>\alpha t^{-1/2}_{C,1}\Biggl\}=0.
\end{align*}
It is easy to see that for all $\alpha$ and $c_1>0$
$$
\limsup_{C\to \infty} \max_{t^*-t_1\leq t \leq t^*-t_{C,1}}\left( \alpha t^{1/2}_{C,1} (t^* -t)-c_1(t^*-t)\Delta  \right)<0,
$$
if $\Delta=O(1)$. (Similar argument can be used when $\Delta\to \infty$.) Similarly to \eqref{noabs-1} we have that
\begin{align}\label{noabs-2}
\lim_{C\to \infty}\liminf_{T\to \infty}P\Biggl\{&\max_{t^*-t_1\leq t\leq t^*- t_{C,1}}\left|\sum_{s=1}^tX_s-\frac{t}{T}\sum_{s=1}^TX_s\right|-\left|\sum_{s=1}^{t^*}X_s-\frac{t^*}{T}\sum_{s=1}^TX_s\right|\\
&=\max_{t^*-t_1\leq t\leq t^*- t_{C,1}}\left(\sum_{s=1}^tX_s-\frac{t}{T}\sum_{s=1}^TX_s\right)-\left(\sum_{s=1}^{t^*}X_s-\frac{t^*}{T}\sum_{s=1}^TX_s\right)
\Biggl\}=1.\notag
\end{align}
Thus we obtain that
$$
\hat{t}_T=\mbox{argmax}_{1\leq t\leq T}\left\{ \left|\sum_{s=1}^tX_s-\frac{t}{T}\sum_{s=1}^TX_s\right|\right\}
$$
satisfies
$
|\hat{t}_T-t^*|I\{\hat{t}_T\leq t^*\}=O_P(\Delta^{-2}).
$
Similarly to \eqref{gc-1} we have
\begin{align}\label{gc-1}
\max_{t^*<t \leq T}(t-t^*)^{-1/2}\left\{\left| \sum_{s=t^*}^{t}e_s   \right|+\frac{t-t^*}{T^{3/2}}\left|\sum_{s=1}^Te_s\right|\right\}=O_P((\log T)^{1/\bar{\nu}}).
\end{align}
Elementary arguments give that
\beq\label{gr-2}
z_{t^*,T}-z_{t,T}\geq (t-t^*) \Delta T^{-3/2}t^*,\quad\mbox{if}\;\;t^*\leq t \leq T.
\eeq
Let $t_{2}=(T(\log T)/(\Delta t^*))^2.$ It follows from \eqref{gr-2} that
\beq\label{neg-1}
\limsup_{T\to \infty}\max_{t^*+t_{2}\leq t\leq T}(T^{1/2}(z_{t,T}-z_{t^*,T})+c_3(t-t^*)^{1/2}(\log T)^{1/\bar{\nu}})= -\infty
\eeq
for all $c_3>0$. Putting together \eqref{gc-1} and \eqref{neg-1} we conclude
$$
\limsup_{T\to \infty}P\left\{\max_{t^*+t_{2}\leq t\leq T }T^{-1/2}\left|\sum_{s=1}^tX_s-\frac{t}{T}\sum_{s=1}^TX_s\right|-T^{-1/2}\left|\sum_{s=1}^{t^*}X_s-\frac{t^*}{T}\sum_{s=1}^TX_s\right|>-x\right\}=0
$$
for all $x>0$. Let $t_{C,2}=(T/(\Delta t^*))^2.$ Using \eqref{rad-4} we conclude that for all $C>0$
\begin{align*}
\lim_{\alpha\to\infty}\limsup_{T\to \infty}P\Biggl\{ \max_{t^*+t_{C,2}\leq t\leq t^*+t_2}\frac{1}{t-t^*}\left|\sum_{s=t^*}^te_s\right|>\alpha t^{1/2}_{C,2}\Biggl\}=0.
\end{align*}
Elementary arguments give that for $C$ sufficiently large
\begin{align*}
\limsup_{T\to \infty}\max_{t^*+t_{2}\leq t\leq T}(T^{1/2}(z_{t,T}-z_{t^*,T})+\alpha(t-t^*)t^{-1/2}_{C,2})= -\infty.
\end{align*}
Similarly to \eqref{noabs-1} and \eqref{noabs-2} we have
\begin{align*}
\lim_{C\to \infty}\liminf_{T\to \infty}P\Biggl\{&\max_{t^*+t_{C,2}\leq t\leq t^*+t_{2}}\left|\sum_{s=1}^tX_s-\frac{t}{T}\sum_{s=1}^TX_s\right|-\left|\sum_{s=1}^{t^*}X_s-\frac{t^*}{T}\sum_{s=1}^TX_s\right|\\
&=\max_{t^*+t_{C,2}\leq t\leq t^*+t_{2}}\left(\sum_{s=1}^tX_s-\frac{t}{T}\sum_{s=1}^TX_s\right)-\left(\sum_{s=1}^{t^*}X_s-\frac{t^*}{T}\sum_{s=1}^TX_s\right)
\Biggl\}=1.\notag
\end{align*}
Thus we conclude
$$
|\hat{t}_T-t^*|I\{\hat{t}_T\geq t^*\}=O_P\left(\left(\frac{T}{\Delta t^*}\right)^2\right).
$$
Hence
$$
\lim_{C\to \infty}\liminf_{T\to\infty}P\left\{ \max_{1\leq t \leq T}\left|\sum_{s=1}^tX_s-\frac{t}{T}\sum_{s=1}^TX_s\right|=\max_{t^*-t_{C,1}\leq t \leq t^*+t_{C,2}} \left|\sum_{s=1}^tX_s-\frac{t}{T}\sum_{s=1}^TX_s\right| \right\}=1.
$$
Also,
\begin{align*}
\lim_{T\to\infty}P\left\{\max_{t^*-t_{C,1}\leq t \leq t^*+t_{C,2}}\left|\sum_{s=1}^{t^*}e_s-\frac{t^*}{T}\sum_{s=1}^Te_s-T^{1/2}z_{t,T}\right|=-\sum_{s=1}^{t^*}e_s+\frac{t^*}{T}\sum_{s=1}^Te_s+T^{1/2}z_{t^*,T}\right\}=1
\end{align*}
 for all $C>0$. It follows from Assumption 2.2 and the stationarity of $e_s, -\infty<s<\infty$ that for all $C>0$
 $$
 \max_{t^*-t_{C,1}\leq t \leq t^*}\left|\sum_{s=t}^{t^*}e_s\right| =O_P\left(\frac{1}{\Delta} \right)
 \quad
 \mbox{and}
 \quad
 \max_{t^*\leq t \leq t^*+t_{C,2}}\left|\sum_{s=t^*}^te_s\right| =O_P\left(\frac{T}{\Delta t^*} \right).
 $$
  Using again Assumption 2.2
 and $t^*/T\to 0$ we conclude
 $$
 \frac{1}{\sqrt{t^*}}\left(\sum_{s=1}^{t^*}e_s-\frac{t^*}{T}\sum_{s=1}^Te_s\right)\;\;\stackrel{{\mathcal D}}{\to}\;\;\sigma{\mathcal N},
 $$
 where ${\mathcal N}$ denotes a standard normal random variable. The result in Theorem 4.2(i) is now proven, since
 $$
 \frac{1}{\sqrt{t^*}}\left(\frac{T}{\Delta t^*}+\frac{1}{\Delta}\right)\to 0.
 $$
 \\

We note that
$$
\lim_{T\to \infty}P\left\{D_T=\sup_{t_T\leq t \leq T-t_T}\left(-\frac{1}{t}\sum_{s=1}^te_s+\frac{1}{T-t}\sum_{s=t+1}^Te_s+ t_T^{-1/2}|v_{t,T}|\right)\right\}=1,
$$
where $v_{t,T}$ is defined in \ref{thm-fix}. We write
$$
\sup_{t_T\leq t \leq T-t_T}\left(-\frac{1}{t}\sum_{s=1}^te_s+\frac{1}{T-t}\sum_{s=t+1}^Te_s+ t_T^{-1/2}|v_{t,T}|\right) -|\Delta| = \max\{G_{T,1},G_{T,2}\},
$$
where

$$
G_{T,1}= \sup_{t_T\leq t \leq t^*}\left(-\frac{1}{t}\sum_{s=1}^te_s+\frac{1}{T-t}\sum_{s=t+1}^Te_s+ t_T^{-1/2} \bar{v}_{t,T}\right),
$$

and

$$
G_{T,2} = \sup_{t^* \leq t \leq T-t_T}\left(-\frac{1}{t}\sum_{s=1}^t e_s+\frac{1}{T-t}\sum_{s=t+1}^T e_s+ t_T^{-1/2}\bar{v}_{t,T}\right),
$$
with $\bar{v}_{t,T}=|v_{t,T}| - t_T^{1/2}|\Delta|$. It follows from the proof of Lemma \ref{lem-fir} that $t_T^{1/2} G_{T,1} = t_T^{1/2} G_{T,3} + o_P(1)$, where

$$
G_{T,3}=  \sup_{t_T \leq t \leq t^*}\left(-\frac{1}{t}\sum_{s=1}^te_s + t_T^{-1/2}\bar{v}_{t,T}\right).
$$
The assumptions of Theorem 4.2(ii) imply that
$$
\max_{t_T \le t \le t^*} | \bar{v}_{t,T}| = O\left( \frac{t_T^{1/2}t^* |\Delta|}{T}\right) = o(1),
$$
and therefore $t_T^{1/2} G_{T,3} = t_T^{1/2} \max_{t_T \le t \le t^* } -\frac{1}{t}\sum_{s=1}^te_s + o_P(1). $ Using the proof of Theorem 2.1 we obtain that

$$
 t_T^{1/2} \max_{t_T \le t \le t^* } -\frac{1}{t}\sum_{s=1}^te_s \stackrel{D}{\to} \sigma \sup_{ 0 \le u \le 1} W(u),
$$
where $W(u)$ is a standard Wiener process. Also, by Theorem 2.1, we have that

$$
t_T^{1/2} \sup_{t^* \le t \le 2t^*} \left(  -\frac{1}{t}\sum_{s=1}^te_s + \frac{1}{T-t}\sum_{s=t+1}^Te_s \right) = o_P(1),
$$
and
\begin{align}\label{thm-fix1}
  t_T^{1/2} \sup_{2t^* \le t \le T-t_T} \left(  -\frac{1}{t}\sum_{s=1}^te_s + \frac{1}{T-t}\sum_{s=t+1}^Te_s \right) = O_P(1).
\end{align}
Since $\bar{v}_{t,T} \le 0$ for all $t$, we have that

\begin{align}\label{thm-fix2}
 t_T^{1/2} \sup_{t^* \le t \le 2t^*} \left(  -\frac{1}{t}\sum_{s=1}^te_s + \frac{1}{T-t}\sum_{s=t+1}^Te_s + t^{-1/2} \bar{v}_{t,T}\right) = o_P(1).
\end{align}
Again due to the assumptions of Theorem 4.2(ii) we have that  $\sup_{2t^* \le t \le T-t_T} \bar{v}_{t,T} \to -\infty$ as $T \to \infty$, and therefore by \eqref{thm-fix1} we have

$$
t_T^{1/2} \sup_{2t^* \le t \le T-t_T} \left(  -\frac{1}{t}\sum_{s=1}^te_s + \frac{1}{T-t}\sum_{s=t+1}^Te_s + t_T^{1/2}\bar{v}_{t,T} \right) \stackrel{P}{\longrightarrow} -\infty,
$$
as $T \to \infty$. Hence

$$
P\left\{ t_T^{1/2} G_{T,2} = \sup_{t^* \le t \le 2t^*} t_T^{1/2} \left( -\frac{1}{t}\sum_{s=1}^te_s + \frac{1}{T-t}\sum_{s=t+1}^Te_s + t_T^{1/2}\bar{v}_{t,T} \right) \right\}=1.
$$
It now follows from \eqref{thm-fix2} that $t_T^{1/2} G_{T,2} = o_P(1)$, from which the result follows.
\qed

\noindent
{\it Proof of Theorem 4.3.} It follows from the definition of the residuals in equation (3.1) of the main text that
 \begin{align*}
 &\frac{1}{t^*}\sum_{s=1}^{t^*}\hat{e}_s-\frac{1}{T-t^*}\sum_{s=t^*+1}^T\hat{e}_s \\
 &=\frac{1}{t^*}\sum_{s=1}^{t^*}{e}_s-\frac{1}{T-t^*}\sum_{s=t^*+1}^T{e}_s +
 \left(\frac{1}{t^*}\sum_{s=1}^{t^*}\bx_s\right)^\T(\bbe_{(1)}-\hat{\bbe}_T)-\left(\frac{1}{T-t^*}\sum_{s=t^*+1}^T\bx_s \right)^\T(\bbe_{(2)}-\hat{\bbe}_T)
 \end{align*}
 and
 $$
 t_T^{1/2}\left(\frac{1}{t^*}\sum_{s=1}^{t^*}{e}_s-\frac{1}{T-t^*}\sum_{s=t^*+1}^T{e}_s\right)=O_P(1)
 $$
 by Theorem 3.1.  Using Assumption 3.1 one can verify that
 $$
 \hat{\bbe}_N=\frac{t^*}{T}\bbe_{(1)}+\frac{T-t^*}{T}\bbe_{(2)}+O_P\left(\frac{\sqrt{t^*}}{T}+\frac{\sqrt{T-t^*}}{T}\right).
 $$
 Hence Assumption  3.1(ii) yields
 \begin{align*}
 r^{1/2}_T&\left\{\left(\frac{1}{t^*}\sum_{s=1}^{t^*}\bx_s\right)^\T(\bbe_{(1)}-\hat{\bbe}_T)-\left(\frac{1}{T-t^*}\sum_{s=t^*+1}^T\bx_s \right)^\T(\bbe_{(2)}-\hat{\bbe}_T)\right\}\\
 &=t_T^{1/2}\left\{\bar{\bx}_0^\T (\bbe_{(1)}-\hat{\bbe}_T)-\bar{\bx}_0^\T (\bbe_{(2)}-\hat{\bbe}_T) \right\}(1+o_P(1)),
 \end{align*}
 so the result follows from Assumption 4.1(ii).
\qed

\section{Multivariate generalization of Theorem 2.1 }\label{amult-gen}

Consider a stationary sequence of random vectors ${\bf X}_t  \in \mathbb{R}^d$ that satisfies which

\begin{align}\label{amodel-m}
 {\bf X}_t={\bm \mu}+\bve_t,\quad 1\leq t \leq T,
 \end{align}
where $E\bve_t = {\bm 0}$, and
$$
\Sigma = \lim_{n\to \infty} \frac{1}{n}\mbox{cov}\left( \sum_{j=1}^{n} \bve_j, \sum_{j=1}^{n} \bve_j \right),
$$
which we assume is well defined and invertible. We let $\| \cdot\|$ denote the Euclidean norm in $\mathbb{R}^d$.  We consider the asymptotic properties of $D_T$ defined by

$$
D_T^2= \sup_{t_T\leq x \leq T-t_T} \left( \frac{1}{\lfloor x \rfloor} \sum_{j=1}^{\lfloor x \rfloor} {\bm X}_j - \frac{1}{T -\lfloor x \rfloor} \sum_{j=\lfloor x \rfloor + 1}^{T} {\bm X}_j  \right)^\top \Sigma^{-1}   \left( \frac{1}{\lfloor x \rfloor} \sum_{j=1}^{\lfloor x \rfloor} {\bm X}_j - \frac{1}{T -\lfloor x \rfloor} \sum_{j=\lfloor x \rfloor + 1}^{T} {\bm X}_j  \right),
$$

i.e. $D_T$ is the square root of the trimmed maximally selected quadratic form based the difference between the averages of ${\bm X}_t$.

\begin{assumption}\label{ano-and-a}
$t_T\to\infty\;\;\mbox{\rm{and}}\;\; t_T/T\to 0,\;\;\mbox{as}\;\;T\to \infty.$
\end{assumption}
In order to establish the limit distribution of $D_T$ under Assumption \ref{ano-and-a}, we require a rate in the weak convergence of the partial sum process of the $\bve_j$'s, which we quantify with the following assumption:
\begin{assumption}\label{aappr-a}  {\rm There are two independent sequences of standard $d-$dimensional Wiener processes $\{\bW_{T,1}(x), 0\leq x \leq T/2\}$ and $\{\bW_{T,2}(x), 0\leq x\leq T/2\}$ such that
\beq\label{an-ap-1}
\max_{1\leq x\leq T/2}x^{-\kappa}\left\|\sum_{s=1}^{\lf x \rf}\bve_s-\Sigma^{1/2} \bW_{T,1}(x)\right\|=O_P(1)
\eeq
and
\beq\label{an-ap-2}
\max_{T/2\leq x\leq T-1}(T-x)^{-\kappa}\left\|\sum_{s=\lf x \rf +1}^{T}\bve_s-\Sigma^{1/2} \bW_{T,2}(T-x)\right\|=O_P(1)
\eeq
with some $0<\kappa<1/2$. }
\end{assumption}

\begin{theorem}\label{amain-a-4} Under these assumptions, we have that
$$
t_T^{1/2}D_T \;\;\stackrel{{\mathcal D}}{\to}\;\;\max(\xi_1, \xi_2),
$$
where  $\xi_1, \xi_2$ are independent and each have the same distribution as
$$\xi \stackrel{D}{=} \sup_{0 \le x \le 1} \| \bW(x)\|,
$$
where $\bW(x)$ is a standard $d$-dimensional Brownian motion.
\end{theorem}

\subsection{\it Proof of Theorem \ref{amain-a-4} }
Evidently under \eqref{amodel-m}, the statistic $D_T$ does not depend on ${\bm \mu}$, and hence we may assume without loss of generality that $D_T$ is defined by
$$
D_T^2= \sup_{t_T\leq x \leq T-t_T} \left( \frac{1}{\lfloor x \rfloor} \sum_{j=1}^{\lfloor x \rfloor} \bve_j - \frac{1}{T -\lfloor x \rfloor} \sum_{j=\lfloor x \rfloor + 1}^{T} \bve_j  \right)^\top \Sigma^{-1}   \left( \frac{1}{\lfloor x \rfloor} \sum_{j=1}^{\lfloor x \rfloor} \bve_j - \frac{1}{T -\lfloor x \rfloor} \sum_{j=\lfloor x \rfloor + 1}^{T} \bve_j  \right).
$$
It follows then that
\begin{align}\label{abasic-1}
 \sup_{t_T\leq x \leq T-t_T} \left( \frac{1}{\lfloor x \rfloor} \sum_{j=1}^{\lfloor x \rfloor} \bve_j - \frac{1}{T -\lfloor x \rfloor} \sum_{j=\lfloor x \rfloor + 1}^{T} \bve_j  \right)^\top & \Sigma^{-1}   \left( \frac{1}{\lfloor x \rfloor} \sum_{j=1}^{\lfloor x \rfloor} \bve_j - \frac{1}{T -\lfloor x \rfloor} \sum_{j=\lfloor x \rfloor + 1}^{T} \bve_j  \right)  \\
 &= \max(V_{T,1}, V_{T,2}) \notag
\end{align}
 where
\begin{align*}
 V_{T, 1}=\sup_{t_T\leq x \leq T/2}  \left( \frac{1}{\lfloor x \rfloor} \sum_{j=1}^{\lfloor x \rfloor} \bve_j - \frac{1}{T -\lfloor x \rfloor} \sum_{j=\lfloor x \rfloor + 1}^{T} \bve_j  \right)^\top \Sigma^{-1}   \left( \frac{1}{\lfloor x \rfloor} \sum_{j=1}^{\lfloor x \rfloor} \bve_j - \frac{1}{T -\lfloor x \rfloor} \sum_{j=\lfloor x \rfloor + 1}^{T} \bve_j  \right),\;\;\; \mbox{and}  \\
 V_{T, 2}=\sup_{T/2 \leq x \leq T-t_T}  \left( \frac{1}{\lfloor x \rfloor} \sum_{j=1}^{\lfloor x \rfloor} \bve_j - \frac{1}{T -\lfloor x \rfloor} \sum_{j=\lfloor x \rfloor + 1}^{T} \bve_j  \right)^\top \Sigma^{-1}   \left( \frac{1}{\lfloor x \rfloor} \sum_{j=1}^{\lfloor x \rfloor} \bve_j - \frac{1}{T -\lfloor x \rfloor} \sum_{j=\lfloor x \rfloor + 1}^{T} \bve_j  \right)
\end{align*}

 \begin{lemma}\label{alem-fir} If  Assumptions \ref{ano-and-a} and \ref{aappr-a} hold, then we have
 \beq\label{afir-1}
 t_T V_{T, 1}=t_T\sup_{t_T\leq x \leq T/2} \left( \frac{1}{\lfloor x \rfloor} \sum_{j=1}^{\lfloor x \rfloor} \bve_j   \right)^\top \Sigma^{-1}   \left( \frac{1}{\lfloor x \rfloor} \sum_{j=1}^{\lfloor x \rfloor} \bve_j  \right) +o_P(1)
 \eeq
 and
 \beq\label{afir-2}
 t_T V_{T, 2}= t_T \sup_{T/2 \leq x \leq T-t_T}  \left(  \frac{1}{T -\lfloor x \rfloor} \sum_{j=\lfloor x \rfloor + 1}^{T} \bve_j  \right)^\top \Sigma^{-1}   \left( \frac{1}{T -\lfloor x \rfloor} \sum_{j=\lfloor x \rfloor + 1}^{T} \bve_j  \right)
+o_P(1).
 \eeq
 \end{lemma}
 \begin{proof}
Let $V_{T,1}^*=  \sup_{t_T\leq x \leq T/2} \left( \frac{1}{\lfloor x \rfloor} \sum_{j=1}^{\lfloor x \rfloor} \bve_j   \right)^\top \Sigma^{-1}   \left( \frac{1}{\lfloor x \rfloor} \sum_{j=1}^{\lfloor x \rfloor} \bve_j  \right)$. It follows that

\begin{align*}
| V_{T,1}-V_{T,1}^* | \le 2  \sup_{t_T\leq x \leq T/2}  &  \Biggl| \Biggl( \frac{1}{T -\lfloor x \rfloor} \sum_{j=\lfloor x \rfloor + 1}^{T} \bve_j   \Biggl)^\top \Sigma^{-1}   \Biggl( \frac{1}{\lfloor x \rfloor} \sum_{j=1}^{\lfloor x \rfloor} \bve_j  \Biggl) \Biggl|  \\
& + \sup_{t_T\leq x \leq T/2}   \Biggl| \Biggl( \frac{1}{T -\lfloor x \rfloor} \sum_{j=\lfloor x \rfloor + 1}^{T} \bve_j   \Biggl)^\top \Sigma^{-1}   \Biggl( \frac{1}{T -\lfloor x \rfloor} \sum_{j=\lfloor x \rfloor + 1}^{T} \bve_j \Biggl) \Biggl| \\
&=: 2E_1 + E_2.
\end{align*}
 The Cauchy-Schwarz inequality implies that

\begin{align}\label{ae1-bound}
  E_1 & \le \| \Sigma^{-1} \|_{op} \sup_{t_T\leq x \leq T/2}   \Biggl\| \frac{1}{T -\lfloor x \rfloor} \sum_{j=\lfloor x \rfloor + 1}^{T} \bve_j  \Biggl\|   \Biggl\| \frac{1}{\lfloor x \rfloor} \sum_{j=1}^{\lfloor x \rfloor} \bve_j  \Biggl\| \\
  & \le \| \Sigma^{-1} \|_{op}  \sup_{t_T\leq x \leq T/2}   \Biggl\| \frac{1}{T -\lfloor x \rfloor} \sum_{j=\lfloor x \rfloor + 1}^{T} \bve_j  \Biggl\|  \sup_{t_T\leq x \leq T/2}  \Biggl\| \frac{1}{\lfloor x \rfloor}   \sum_{j=1}^{\lfloor x \rfloor} \bve_j  \Biggl\|, \notag
\end{align}
Where $\|\cdot \|_{op}$ denotes the operator norm of a matrix. We first aim to bound the second term on the right hand side of the last line. We have for this term that

$$
\sup_{t_T\leq x \leq T/2}   \Biggl\| \frac{1}{T -\lfloor x \rfloor} \sum_{j=\lfloor x \rfloor + 1}^{T} \bve_j  \Biggl\| \le \frac{2}{T}  \sup_{t_T\leq x \leq T/2}   \Biggl\|  \sum_{j=\lfloor x \rfloor + 1}^{T} \bve_j  \Biggl\|.
$$
By the triangle inequality it follows that

$$
\sup_{t_T\leq x \leq T/2}   \Biggl\|  \sum_{j=\lfloor x \rfloor + 1}^{T} \bve_j  \Biggl\| \le \sup_{t_T\leq x \leq T/2}   \Biggl\|  \sum_{j=1}^{\lfloor x \rfloor} \bve_j  \Biggl\| + \Biggl\|  \sum_{j=1}^{T/2} \bve_j  \Biggl\| + \Biggl\|  \sum_{j=T/2+1}^{T} \bve_j  \Biggl\|.
$$
Again using the triangle inequality, we obtain that
$$
\sup_{t_T\leq x \leq T/2}   \Biggl\|  \sum_{j=1}^{\lfloor x \rfloor} \bve_j  \Biggl\|  \le \sup_{t_T\leq x \leq T/2}   \Biggl\|  \sum_{j=1}^{\lfloor x \rfloor} \bve_j - \Sigma^{1/2}\bW_{T,2}(x) \Biggl\|  + \sup_{t_T\leq x \leq T/2}   \Biggl\|  \Sigma^{1/2}\bW_{T,2}(x)  \Biggl\|.
$$
By Assumption \ref{ano-and-a},
$$
\sup_{t_T\leq x \leq T/2}   \Biggl\|  \sum_{j=1}^{\lfloor x \rfloor} \bve_j - \Sigma^{1/2}\bW_{T,2}(x) \Biggl\| = O_P(T^{\kappa}),
$$
and by the scale transformation of the Brownian motion,

$$
\sup_{t_T\leq x \leq T/2}   \Biggl\|  \Sigma^{1/2}\bW_{T,2}(x)  \Biggl\| = O_P(T^{1/2}).
$$
It follows similarly from Assumption \ref{ano-and-a} that

$$
\Biggl\|  \sum_{j=1}^{T/2} \bve_j  \Biggl\|=O_P(T^{1/2}),\;\mbox{and} \;  \Biggl\|  \sum_{j=T/2+1}^{T} \bve_j  \Biggl\|=O_P(T^{1/2}),
$$

and therefore $\sup_{t_T\leq x \leq T/2}   \Big\| \frac{1}{T -\lfloor x \rfloor} \sum_{j=\lfloor x \rfloor + 1}^{T} \bve_j \Big\| = O_P(T^{-1/2})$. We now turn to bounding the third term on the last line of \eqref{ae1-bound}. We have for this term by applying Assumption \ref{ano-and-a}, the triangle inequality, and the scale transformation of the Brownian motion that

\begin{align*}
\sup_{t_T\leq x \leq T/2}  \Biggl\| \frac{1}{\lfloor x \rfloor}  \sum_{j=1}^{\lfloor x \rfloor} \bve_j  \Biggl\| & \le   \sup_{t_T\leq x \leq T/2}   \frac{1}{\lfloor x \rfloor} \Biggl\|   \sum_{j=1}^{\lfloor x \rfloor} \bve_j - \Sigma^{1/2} \bW_{T,2}(x)  \Biggl\| + \sup_{t_T\leq x \leq T/2}   \frac{1}{\lfloor x \rfloor} \Biggl\|   \Sigma^{1/2} \bW_{T,2}(x)  \Biggl\| \\
& \le  \frac{1}{t_T^{1-\kappa}} \sup_{t_T\leq x \leq T/2}   \frac{1}{\lfloor x \rfloor^{\kappa}} \Biggl\|   \sum_{j=1}^{\lfloor x \rfloor} \bve_j - \Sigma^{1/2} \bW_{T,2}(x)  \Biggl\| + \max_{1\leq y \leq T/(2t_T)}  \frac{1}{t_T^{1/2}} \Biggl\| \Sigma^{1/2} \frac{\bW_{T,2}(y)}{y} \Biggl\| \\
&= O_P\left( \frac{1}{t_T^{1-\kappa}} +\frac{1}{t_T^{1/2}} \right)=  O_P(t_T^{-1/2}).
\end{align*}
It now follows that $t_TE_1=O_P( (t_T/T)^{1/2})=o_P(1)$.  One may obtain by similar means that $t_TE_2 = o_P(1)$, from which \eqref{afir-1} follows. The proof of \eqref{afir-2} follows similar lines, and so we omit the details.

\end{proof}

 \begin{lemma}\label{awe-le} If  Assumptions \ref{ano-and-a} and \ref{aappr-a} hold, then we have
 \begin{align*}
 \max\Biggl( t_T  &\sup_{t_T\leq x \leq T/2} \Biggl( \frac{1}{\lfloor x \rfloor} \sum_{j=1}^{\lfloor x \rfloor} \bve_j   \Biggl)^\top \Sigma^{-1}   \Biggl( \frac{1}{\lfloor x \rfloor} \sum_{j=1}^{\lfloor x \rfloor} \bve_j  \Biggl) , \; \\ t_T &\sup_{T/2 \leq x \leq T-t_T}  \Biggl(  \frac{1}{T -\lfloor x \rfloor} \sum_{j=\lfloor x \rfloor + 1}^{T} \bve_j  \Biggl)^\top \Sigma^{-1}   \Biggl( \frac{1}{T -\lfloor x \rfloor} \sum_{j=\lfloor x \rfloor + 1}^{T} \bve_j  \Biggl) \Biggl)\;\;\stackrel{{\mathcal D}}{\to}\;\sigma \max (\xi_1, \xi_2),
 \end{align*}

 where $\xi_1$ and $\xi_2$ are defined in Theorem \ref{amain-a-4}.
 \end{lemma}

 \begin{proof}

 We first aim to show that

\begin{align}\label{afir-1-1}
  t_T \Biggl| \sup_{t_T\leq x \leq T/2} \Biggl(& \frac{1}{\lfloor x \rfloor}  \sum_{j=1}^{\lfloor x \rfloor} \bve_j   \Biggl)^\top \Sigma^{-1}   \Biggl( \frac{1}{\lfloor x \rfloor} \sum_{j=1}^{\lfloor x \rfloor} \bve_j  \Biggl)  \\
  -& \sup_{t_T\leq x \leq T/2} \Biggl( \frac{\Sigma^{1/2}}{\lfloor x \rfloor} \bW_{T,1}(x)   \Biggl)^\top \Sigma^{-1}   \Biggl( \frac{\Sigma^{1/2}}{\lfloor x \rfloor} \bW_{T,1}(x)  \Biggl) \Biggl| = o_P(1), \notag
\end{align}
and

\begin{align}\label{afir-2-2}
  t_T \Biggl| \sup_{T/2 \leq x \leq T-t_T} \Biggl(& \frac{1}{T -\lfloor x \rfloor}  \sum_{j=\lfloor x \rfloor + 1}^{T} \bve_j   \Biggl)^\top \Sigma^{-1}   \Biggl( \frac{1}{T -\lfloor x \rfloor}  \sum_{j=\lfloor x \rfloor + 1}^{T} \bve_j   \Biggl)  \\ -& \sup_{T/2 \leq x \leq T-t_T}  \Biggl( \frac{\Sigma^{1/2}}{\lfloor x \rfloor} \bW_{T,2}(x)   \Biggl)^\top \Sigma^{-1}   \Biggl( \frac{\Sigma^{1/2}}{\lfloor x \rfloor} \bW_{T,2}(x)  \Biggl) \Biggl| = o_P(1). \notag
\end{align}

Towards establishing \eqref{afir-1-1}, we note that the left hand side is bounded above by
$$
t_T  \sup_{t_T\leq x \leq T/2} \Biggl| \Biggl( \frac{1}{\lfloor x \rfloor}  \sum_{j=1}^{\lfloor x \rfloor} \bve_j   \Biggl)^\top \Sigma^{-1}   \Biggl( \frac{1}{\lfloor x \rfloor} \sum_{j=1}^{\lfloor x \rfloor} \bve_j  \Biggl) -\Biggl( \frac{\Sigma^{1/2}}{\lfloor x \rfloor} \bW_{T,1}(x)   \Biggl)^\top \Sigma^{-1}   \Biggl( \frac{\Sigma^{1/2}}{\lfloor x \rfloor} \bW_{T,1}(x)  \Biggl) \Biggl|,
$$
which by the triangle inequality is less than or equal to

\begin{align*}
t_T  \sup_{t_T\leq x \leq T/2} \Biggl| \Biggl( & \frac{1}{\lfloor x \rfloor}  \sum_{j=1}^{\lfloor x \rfloor} \bve_j  - \frac{\Sigma^{1/2}}{\lfloor x \rfloor} \bW_{T,1}(x)  \Biggl)^\top \Sigma^{-1}   \Biggl( \frac{1}{\lfloor x \rfloor} \sum_{j=1}^{\lfloor x \rfloor} \bve_j  \Biggl) \Biggl| \\
&+ t_T  \sup_{t_T\leq x \leq T/2}  \Biggl| \Biggl( \frac{1}{\lfloor x \rfloor}  \sum_{j=1}^{\lfloor x \rfloor} \bve_j    \Biggl)^\top \Sigma^{-1}   \Biggl( \frac{1}{\lfloor x \rfloor}  \sum_{j=1}^{\lfloor x \rfloor} \bve_j - \frac{\Sigma^{1/2}}{\lfloor x \rfloor} \bW_{T,1}(x)  \Biggl) \Biggl|=: G_1 + G_2.
\end{align*}

For the first summand we have by the Cauchy-Schwarz inequality that

\begin{align*}
t_T  \sup_{t_T\leq x \leq T/2} \Biggl| \Biggl( & \frac{1}{\lfloor x \rfloor}  \sum_{j=1}^{\lfloor x \rfloor} \bve_j  - \frac{\Sigma^{1/2}}{\lfloor x \rfloor} \bW_{T,1}(x)  \Biggl)^\top \Sigma^{-1}   \Biggl( \frac{1}{\lfloor x \rfloor} \sum_{j=1}^{\lfloor x \rfloor} \bve_j  \Biggl) \Biggl| \\
& \le  t_T   \| \Sigma^{-1}\|_{op} \sup_{t_T\leq x \leq T/2}  \Biggl\| \frac{1}{\lfloor x \rfloor}  \sum_{j=1}^{\lfloor x \rfloor} \bve_j - \frac{\Sigma^{1/2}}{\lfloor x \rfloor} \bW_{T,1}(x) \Biggl\| \sup_{t_T\leq x \leq T/2}  \Biggl\|  \frac{1}{\lfloor x \rfloor} \sum_{j=1}^{\lfloor x \rfloor} \bve_j  \Biggl\|.
\end{align*}
It now follows as in the proof of Lemma \ref{alem-fir} that

$$
\sup_{t_T\leq x \leq T/2}  \Biggl\| \frac{1}{\lfloor x \rfloor}  \sum_{j=1}^{\lfloor x \rfloor} \bve_j - \frac{\Sigma^{1/2}}{\lfloor x \rfloor} \bW_{T,1}(x) \Biggl\| = O_P(t_T^{1-\kappa}), \;\mbox{and}\; \sup_{t_T\leq x \leq T/2}  \Biggl\|  \frac{1}{\lfloor x \rfloor} \sum_{j=1}^{\lfloor x \rfloor} \bve_j  \Biggl\|=O_P(t_T^{-1/2}),
$$

from which we obtain that $G_1=O_P(t_T^{\kappa-1/2})=o_P(1)$. A parallel argument shows that $G_2=o_P(1)$, which implies \eqref{afir-1-1}, and by the same argument one can establish \eqref{afir-2-2}. Now we aim to show that

$$
t_T\sup_{t_T\leq x \leq T/2} \Biggl( \frac{\Sigma^{1/2}}{\lfloor x \rfloor} \bW_{T,1}(x)   \Biggl)^\top \Sigma^{-1}   \Biggl( \frac{\Sigma^{1/2}}{\lfloor x \rfloor} \bW_{T,1}(x)  \Biggl)  \stackrel{D}{\to} \xi_1,\;\; T\to \infty.
$$
By simple matrix algebra and the scale transformation of the Brownian motion, we have that

\begin{align*}
 t_T \sup_{t_T\leq x \leq T/2} \Biggl( \frac{\Sigma^{1/2}}{\lfloor x \rfloor} \bW_{T,1}(x)   \Biggl)^\top \Sigma^{-1}  & \Biggl( \frac{\Sigma^{1/2}}{\lfloor x \rfloor} \bW_{T,1}(x)  \Biggl)  \\
 &= t_T \sup_{t_T\leq x \leq T/2} \Biggl\| \frac{\bW_{T,1}(x)}{\lfloor x \rfloor} \Biggl\|^2 \\
 &= t_T  \max_{1\leq y \leq T/(2t_T)}\Biggl\| \frac{\bW_{T,1}(t_Ty)}{t_Ty} \Biggl\|^2 \\
 & =  \max_{1\leq y \leq T/(2t_T)}\Biggl\| \frac{\bW_{T,1}(y)}{y} \Biggl\|^2 \stackrel{D}{\to} \sup_{1\leq y \infty}\Biggl\| \frac{\bW_{0,1}(y)}{y} \Biggl\|^2,
\end{align*}
as $T\to \infty$. A simple calculation shows that  $\{ {\bW_{0,1}(y)}/{y} \;: y\in [1,\infty)\}\stackrel{D}{=} \{ {\bW_{0,1}(1/y)} \;: y\in [1,\infty)\}$, and so we have that

$$
t_T \sup_{t_T\leq x \leq T/2} \Biggl( \frac{\Sigma^{1/2}}{\lfloor x \rfloor} \bW_{T,1}(x)   \Biggl)^\top \Sigma^{-1}   \Biggl( \frac{\Sigma^{1/2}}{\lfloor x \rfloor} \bW_{T,1}(x)  \Biggl) \stackrel{D}{\to} \xi_1.
$$
A similar argument gives that

$$
t_T\sup_{T/2 \leq x \leq T-t_T}  \Biggl( \frac{\Sigma^{1/2}}{\lfloor x \rfloor} \bW_{T,2}(x)   \Biggl)^\top \Sigma^{-1}   \Biggl( \frac{\Sigma^{1/2}}{\lfloor x \rfloor} \bW_{T,2}(x)  \Biggl) \stackrel{D}{\to} \xi_2,
$$
where $\xi_1$ and $\xi_2$ are independent, which proves the result.

\end{proof}
 {\it Proof of Theorem \ref{amain-a-4}.} It follows immediately from Lemmas \ref{alem-fir} and \ref{awe-le}.
 \qed

\section{Asymmetric Trimming}

One can obtain similar results as Theorem 2.1 in the case when $D_T$ is defined with asymmetric trimming.

The asymptotic distribution of $t_T^{1/2}D_T$ is established by means of the following, somewhat more general, result. Let $s_T$ be a sequence satisfying the following assumption.
\begin{assumption}\label{ano-and-a-2}\;\;
$s_T\to \infty\quad \mbox{{\rm and}}\;\; s_T/T\to \infty \;\;\mbox{as}\;\;T\to \infty.$
\end{assumption}
Let
\beq\label{ar-def}
r_T=\min(t_T, s_T)
\eeq
and define
\beq\label{agam-def}
\lim_{T\to\infty}\frac{r_T}{t_T}=\gamma_1\quad\mbox{and}\quad \lim_{T\to\infty}\frac{r_T}{s_T}=\gamma_2.
\eeq
The limit distribution will be expressed in terms of the random variable
\beq\label{axidef}
\xi=\sup_{0\leq u \leq 1}|W(u)|,
\eeq
where $\{W(u), 0\leq u \leq 1\}$ denotes a Wiener process.

\begin{theorem}\label{amain-a-trim} If $H_0$, the conditions of Theorem 2.1, Assumption \ref{ano-and-a-2} hold,  then, as $T\to \infty$, we have
$$
\frac{r_T^{1/2}}{\sigma}\max_{t_T\leq t \leq T- s_T}\left|\frac{1}{t}\sum_{s=1}^tX_s-\frac{1}{T-t}\sum_{s=t+1}^TX_s    \right|\;\;\stackrel{{\mathcal D}}{\to}\;\;\max(\gamma^{1/2}_1\xi_1, \gamma^{1/2}_2\xi_2),
$$
where  $\xi_1, \xi_2$ are independent and each have the same distribution as $\xi$ defined in \eqref{axidef}.
\end{theorem}

We can also formulate our results in terms of more generally weighted CUSUM processes. Let

$$
F_T(\tau)=\frac{1}{\sigma }T^{-\tau}r_T^{\tau-1/2}\sup_{t_T/T\leq u \leq 1-s_T/T}(u(1-u))^{-\tau}\left| \sum_{s=1}^{\lf Tu\rf}X_s-\frac{\lf Tu\rf}{T}\sum_{s=1}^TX_s\right|,
$$
where $\tau $ satisfies $1/2 < \tau < \infty$. The limit distribution of $F_T(\tau)$ will be given in terms of the distribution of the random variable
\beq\label{abexi}
\xi(\tau)=\sup_{0\leq u\leq 1}u^{\tau-1}|W(u)|.
\eeq


We note that by the law of iterated logarithm at zero for the Wiener process, the random variable $\xi(\tau)$ is finite a.s.\ for all $1/2<\tau <\infty$.

\begin{theorem}\label{amain-2} If $H_0$ and Assumptions \ref{ano-and-a}--\ref{ano-and-a-2} hold, and $1/2 < \tau < \infty$, then, as $T\to \infty$, we have
$$
F_T(\tau)\;\;\stackrel{{\mathcal D}}{\to}\;\;\max(\gamma^{\tau-1/2}_1\xi_1(\tau), \gamma^{\tau-1/2}_2\xi_2(\tau)),
$$
where $\xi_1(\tau), \xi_2(\tau)$ are independent and they have the same distribution as $\xi(\tau)$.
\end{theorem}

We can also formulate our results in terms of more generally weighted CUSUM processes. Let

$$
F_T(\tau)=\frac{1}{\sigma }T^{-\tau}r_T^{\tau-1/2}\sup_{t_T/T\leq u \leq 1-s_T/T}(u(1-u))^{-\tau}\left| \sum_{s=1}^{\lf Tu\rf}X_s-\frac{\lf Tu\rf}{T}\sum_{s=1}^TX_s\right|,
$$
where $\tau $ satisfies $1/2 < \tau < \infty$. The limit distribution of $F_T(\tau)$ will be given in terms of the distribution of the random variable
\beq\label{abexi}
\xi(\tau)=\sup_{0\leq u\leq 1}u^{\tau-1}|W(u)|.
\eeq


We note that by the law of iterated logarithm at zero for the Wiener process, the random variable $\xi(\tau)$ is finite a.s.\ for all $1/2<\tau <\infty$.

\begin{theorem}\label{amain-2} If $H_0$ and Assumptions \ref{ano-and-a}--\ref{ano-and-a-2} hold, and $1/2 < \tau < \infty$, then, as $T\to \infty$, we have
$$
F_T(\tau)\;\;\stackrel{{\mathcal D}}{\to}\;\;\max(\gamma^{\tau-1/2}_1\xi_1(\tau), \gamma^{\tau-1/2}_2\xi_2(\tau)),
$$
where $\xi_1(\tau), \xi_2(\tau)$ are independent and they have the same distribution as $\xi(\tau)$.
\end{theorem}

To prove Theorem \ref{amain-2},  we write under $H_0$ that
 \begin{align}\label{abasic-2}
 \sup_{t_T/T\leq u\leq 1-t_T/T}(u(1-u))^{-\tau}\left|\sum_{s=1}^{\lf Tu\rf }X_s-\frac{\lf Tu\rf }{T}\sum_{s=1}^T X_s  \right|=\max(U_{T,1}, U_{T,2}),
 \end{align}
 where
 $$
 U_{T,1}=\sup_{t_T/T\leq u\leq 1/2}(u(1-u))^{-\tau}\left|\sum_{s=1}^{\lf Tu\rf}e_s- \frac{\lf Tu\rf }{T}\sum_{s=1}^T e_s\right|,
 $$
  and
 $$
 U_{T,2}=\sup_{1/2< u<1-{s}_T/T}(u(1-u))^{-\tau}\left|\sum_{s=1}^{\lf Tu\rf}e_s- \frac{\lf Tu\rf }{T}\sum_{s=1}^T e_s\right|.
 $$
 \begin{lemma}\label{alem-1} If Assumptions \ref{ano-and-a}--\ref{aappr-a} hold, and $1/2 < \tau <\infty$,  then we have that
 \beq\label{al-1-1}
 T^{-\tau}t_T^{\tau-1/2}U_{T,1}=T^{-\tau}t_T^{\tau-1/2}\sup_{t_T/T\leq u\leq 1/2}(u(1-u))^{-\tau}\left|\sum_{s=1}^{\lf Tu\rf}e_s\right|+o_P(1)
 \eeq
 and
 \beq\label{al-1-2}
 T^{-\tau}t_T^{\tau-1/2}U_{T,2}= \sup_{1/2< u<1-{s}_T/T}(u(1-u))^{-\tau}\left|\sum_{s=\lf Tu\rf+1}^{ T}e_s \right|  +o_P(1).
 \eeq
 \end{lemma}

 \begin{proof} Let
 $$
 U_{T,1,1}=\left|\sup_{t_T/T\leq u \leq 1/2}(u(1-u))^{-\tau}\left|\sum_{s=1}^{\lf Tu\rf}e_s\right| -U_{T,1}\right|,
 $$
 then we get
 \begin{align*}
 U_{T,1,1}\leq \sup_{t_T/T\leq u\leq 1/2}(u(1-u))^{-\tau}\left|\frac{\lf Tu\rf }{T}\sum_{s=1}^T e_s\right|\leq 2^\tau\left|\sum_{s=1}^{T}e_s\right|\sup_{t_T/T\leq u\leq 1/2}u^{1-\tau}.
 \end{align*}
 We showed in the proof of Lemma \ref{alem-fir} that
 \beq\label{as-1}
 \left|\sum_{s=1}^{T}e_s\right|=O_P(T^{1/2})
 \eeq
 and therefore
 \begin{displaymath}
 U_{T,1,1}=\left\{
 \begin{array}{ll}
 O_P(T^{1/2} ({t}_T/T)^{1-\tau}),\quad &\mbox{if} \;\;1\leq \tau<\infty
 \vspace{.3cm}\\
 O_P(T^{1/2}),&\mbox{if}\;\;1/2<\tau<1.
 \end{array}
 \right.
 \end{displaymath}
 Since by Assumption \ref{ano-and-a}  we have
 $$
 T^{-\tau}t_T^{\tau-1/2}T^{1/2} ({t}_T/T)^{1-\tau}=(t_T/T)^{1/2}\to 0,\quad \mbox{if}\;\;1\leq \tau <\infty
 $$
 and
 $$
 T^{-\tau}t_T^{\tau-1/2}T^{1/2}=(t_T/T)^{\tau-1/2}\to 0,\quad\mbox{if}\;\;1/2<\tau<1
 $$
 implying
 $$
 T^{-\tau}t_T^{\tau-1/2}U_{T,1,1}=o_P(1).
 $$
 Hence \eqref{al-1-1} is proven and the same arguments give \eqref{al-1-2}.
 \end{proof}

 \begin{lemma}\label{aconv-1} If Assumptions \ref{ano-and-a}, \ref{aappr-a} and $1/2 < \tau < \infty$ hold and
 \beq\label{abardef}
 \bar{t}_T/t_T\to \infty,
 \eeq
  then we have that
 $$
 t_T^{\tau-1/2}\sup_{t_T\leq x \leq \bar{t}_T}x^{-\tau}\left|\sum_{s=1}^xe_s\right| \;\;\stackrel{{\mathcal D}}{\to}\;\;\sigma\xi(\tau).
 $$
 \end{lemma}
 \begin{proof}
 By Assumption \ref{aappr-a} we have
 \begin{align*}
 t_T^{\tau-1/2}\sup_{t_T\leq x \leq \bar{t}_T}x^{-\tau}\left|\sum_{s=1}^xe_s-\sigma W_{T,1}(x)\right|=O_P(1)t_T^{\tau-1/2}\sup_{t_T\leq x \leq \bar{t}_T}x^{\kappa-\tau}
 = o_P(1).
 \end{align*}
  By the scale transformation of the Wiener process we have
 \begin{align*}
 t_T^{\tau-1/2}\sup_{t_T\leq x \leq \bar{t}_T}x^{-\tau}|W_{T,1}(x)|\stackrel{{\mathcal D}}{=} \sup_{1\leq y \leq \bar{t}_T/t_T}y^{-\tau}|W(y)|,
 \end{align*}
 where $W$ is a Wiener process. It is easy to see that as $T\to \infty$
 $$
 \sup_{1\leq y \leq \bar{t}_T/t_T}y^{-\tau}|W(y)|\;\;\to\;\; \sup_{1\leq y <\infty}y^{-\tau}|W(y)|\;\;\mbox{a.s.}\;\;\mbox{and}
 \sup_{1\leq y <\infty}x^{-\tau}|W(x)|\stackrel{{\mathcal D}}{=}\xi(\tau),
 $$
 completing the proof of the lemma.
 \end{proof}
 \begin{lemma}\label{alemma-5}  We assume that Assumptions \ref{ano-and-a}--\ref{aappr-a} are satisfied and $1/2 < \tau < \infty$.
 If \eqref{abardef} holds,
\beq\label{abardef-1}
\bar{t}_T/T\to 0,\;\;t_T/ \bar{s}_T\;\to \;0\;\;\;\,\mbox{and}\;\;\;\bar{s}_T/T\to 0,
\eeq
 then we have
 \begin{align}\label{alim-1}
 \frac{t_T^{\tau-1/2}}{T^\tau}\sup_{t_T/T\leq u\leq 1/2}(u(1-u))^{-\tau}\left|\sum_{s=1}^{\lf Tu\rf}e_s\right|
 =\frac{t_T^{\tau-1/2}}{T^{\tau}}\sup_{t_T/T\leq u\leq \bar{t}_T/T}u^{-\tau}\left|\sum_{s=1}^{\lf Tu\rf}e_s\right|+o_P(1)
 \end{align}
 and
 \begin{align}\label{alim-2}
 T^{-\tau}t_T^{\tau-1/2}&\sup_{1/2< u\leq 1-t_T/T}(u(1-u))^{-\tau}\left|\sum_{s=1}^{\lf Tu\rf}e_s\right|\\
 &=T^{-\tau}t_T^{\tau-1/2}\sup_{1-\bar{s}_T/T< u\leq 1-{s}_T/T}(1-u)^{-\tau}\left|\sum_{s=\lf Tu\rf +1}^{ T}e_s\right|+o_P(1).\notag
 \end{align}
  \end{lemma}
 \begin{proof}
 It follows from Lemma \ref{aconv-1} that
 \begin{align*}
 T^{-\tau}t_T^{\tau-1/2}\sup_{\bar{t}_T/T\leq u\leq 1/2}(u(1-u))^{-\tau}\left|\sum_{s=1}^{\lf Tu\rf}e_s\right|=O_P((t_T/\bar{t}_T)^{\tau-1/2})=o_P(1)
 \end{align*}
 and by the mean value theorem
 \begin{align*}
 T^{-\tau}t_T^{\tau-1/2}\sup_{{t}_T/T\leq u\leq \bar{t}_T/T}\left|(u(1-u))^{-\tau}-u^{-\tau}\right|\left|\sum_{s=1}^{\lf Tu\rf}e_s\right|
 &=O(1)T^{-\tau}t_T^{\tau-1/2}\sup_{{t}_T/T\leq u\leq \bar{t}_T/T}u^{1-\tau}\left|\sum_{s=1}^{\lf Tu\rf}e_s\right|\\
 &\leq O(1/\bar{t}_T)T^{-\tau}t_T^{\tau-1/2}\sup_{{t}_T/T\leq u\leq \bar{t}_T/T}u^{-\tau}\left|\sum_{s=1}^{\lf Tu\rf}e_s\right|\\
 &=o_P(1).
 \end{align*}
 The result in \eqref{alim-2} can be proven along the lines of that of \eqref{alim-1} and therefore the proof is omitted.
 \end{proof}

 \begin{lemma}\label{afi-weak} If Assumptions \ref{ano-and-a}--\ref{aappr-a}, \eqref{abardef-1} and  \eqref{abardef} hold, then we have that
 \begin{align*}
\Biggl(T^{-\tau}t_T^{\tau-1/2}&\sup_{t_T/T\leq u\leq \bar{t}_T/T}u^{-\tau}\left|\sum_{s=1}^{\lf Tu\rf}e_s\right|, \;\;\;\; T^{-\tau}t_T^{\tau-1/2}\sup_{1-\bar{s}_T/T< u\leq 1-{s}_T/T}(1-u)^{-\tau}\left|\sum_{s=\lf Tu\rf +1}^{ T}e_s\right|\Biggl)\\
&\stackrel{{\mathcal D}}{\to}\;\;\sigma \max(\xi_1(\tau), \xi_2(\tau)),
\end{align*}
where $\xi_1(\tau)$ and $\xi_2(\tau)$ are defined in Theorem \ref{amain-2}.
 \end{lemma}
 \begin{proof} It follows from Assumption \ref{aappr-a} that
 $$
 T^{-\tau}t_T^{\tau-1/2}\sup_{t_T/T\leq u\leq \bar{t}_T/T}u^{-\tau}\left|\sum_{s=1}^{\lf Tu\rf}e_s\right|=t_T^{\tau-1/2}\sup_{t_T\leq x\leq \bar{t}_T}x^{-\tau}\sigma |W_{T,1}(x)|+o_P(1)
$$
and
\begin{align*}
T^{-\tau}t_T^{\tau-1/2}&\sup_{1-\bar{s}_T/T< u\leq 1-{s}_T/T}(1-u)^{-\tau}\left|\sum_{s=\lf Tu\rf +1}^{ T}e_s\right|\\
&=
t_T^{\tau-1/2}\sup_{T-\bar{s}_T< x\leq T-{s}_T}(T-x)^{-\tau}\sigma |W_{T,2}(x)|\\
&=o_P(1).
\end{align*}
The asymptotic independence now follows from the independence of the Wiener processes and the asymptotic distribution is an immediate consequence  of Lemma
\ref{aconv-1}.
 \end{proof}

 \noindent
 {\it Proof of Theorem \ref{amain-2}.} The result is an immediate consequence of Lemmas \ref{alem-1}--\ref{afi-weak}.
 \qed

\section{Estimation of the long run variance: Proof of Theorems 5.1 and consistency of estimators defined in (5.4)}

In this section we provide justification of the Theorem 5.1, and establish the consistency of the estimators defined in (5.4). We begin with Theorem 5.1.

\noindent
{\it Proof of Theorem 5.1.} Elementary arguments show that under $H_0$ we have
\begin{align}\label{atrualt}
T\hat{\sigma}^2_{T,t}=\sum_{s=1}^Te_s^2-t\bar{e}_t^2-(T-t)\tilde{e}_{T-t}^2,\;\; \mbox{with}\;\; \bar{e}_t=\frac{1}{t}\sum_{s=1}^te_s\;\;\mbox{and}
\;\;\tilde{e}_{T-t}=\frac{1}{T-t}\sum_{s=t+1}^Te_s.
\end{align}
The approximations in Assumption \ref{aappr-a} with the Darling--Erd\H{o}s (1956) yields
\begin{align}\label{atrua-1}
\max_{1\leq t \leq T}t|\bar{e}_t^2|=\max_{1\leq t \leq T}\left(\frac{1}{t^{1/2}}\sum_{s=1} e_s \right)^2=O_P(\log\log T)
\end{align}
and
\beq\label{atrua-2}
\max_{1\leq t < T}\left|(T-t)\tilde{e}_{T-t}^2\right|=O_P(\log\log T).
\eeq
By the ergodic theorem,
\beq\label{atrua-3}
\frac{1}{T}\sum_{s=1}^Te_s^2\;\to\;\sigma^2\quad a.s.,
\eeq
completing the proof of (5.1). Observing that \eqref{atrualt} holds true under $H_A$ when $t=t^*$, we get immediately (5.2) from \eqref{atrua-1}--\eqref{atrua-3}.
\qed
\\

\medskip

We now turn to (5.4) from the main paper, in which we use the following assumption:

\begin{assumption}\label{abern-a}

{\rm The sequence $  e_s -\infty< s< \infty  $ is  a Bernoulli shift, i.e. there is measurable function $f$ such that
	  $e_s=f(\vare_s,\vare_{s-1},\ldots)$, where
	 $  \vare_s, -\infty< s< \infty $ are independent and identically distributed random variables in some measurable space. In addition,
$$
Ee_0=0,\; E|e_{0}|^\nu<\infty \;\;\mbox{with some }\;\; \nu>2,
$$
and
$$
\left( {E}\left| e_{s,m}-e_s\right|^{\nu } \right)^{1/\nu}=O(m^{-\gamma})\;\;\mbox{with some}\;\;\gamma>1,
$$
where $e_{s,m} = f(\vare_s,\vare_{s-1},\ldots,\vare_{s-m},\vare^*_{s, m, s-m-1},\vare^*_{s, m, s-m-2},\ldots)$ and $\vare^*_{i,j,\ell}$ are independent and identically distributed  copies of $\vare_0$. }
\end{assumption}

\begin{theorem}\label{asi-th-2}  If Assumption \ref{abern-a} holds with $\nu>4$ moments, and  Assumptions  5.1,  5.2 hold,  then
the estimator $\hat{\sigma}^2_{T,t}$ defined by (5.4) satisfies (5.1) and (5.2).
\end{theorem}

\begin{lemma}\label{api-lem} Suppose Assumption \ref{abern-a} is satisfied with $\nu>4$, and  Assumptions  5.1,  5.2 hold.\\
(i) If $H_0$ holds, then we have
$$
\max_{1\leq t \leq T}\left|\frac{1}{T-1}\sum_{s=1}^Te_s^2-\hat{\gamma}_{\ell,t} \right|=o_P(1)
$$
and
$$
\max_{1\leq t \leq T}\sum_{\ell=1}^{T-1}K\left(\frac{\ell}{h}\right)\left|\frac{1}{T-\ell}\sum_{s=1}^{T-\ell}e_se_{s+\ell}-\hat{\gamma}_{\ell,t} \right|=o_P(1).
$$
(ii)  If $H_A$ holds, then we have
$$
\left|\frac{1}{T}\sum_{s=1}^{T}e_s^2-\hat{\gamma}_{\ell,t^*} \right|=o_P(1)
$$
and
$$
\sum_{\ell=1}^{T-1}K\left(\frac{\ell}{h}\right)\left|\frac{1}{T-\ell}\sum_{s=1}^{T-\ell}e_se_{s+\ell}-\hat{\gamma}_{\ell,t^*} \right|=o_P(1).
$$
\end{lemma}
\begin{proof} First we note that Assumption \ref{aappr-a} and (4.4) hold under the conditions of the  Lemmas (cf.\ Aue et al. (2014)).  Under $H_0$ we have for all $0\leq \ell\leq t<T-\ell$ and $T$
\begin{align}\label{api-dec}
(T-\ell)\hat{\gamma}_{\ell,t}&=\sum_{s=1}^{T-\ell}e_se_{s+\ell}-\bar{e}_t\sum_{s=1}^{t-\ell}e_{s+\ell}-\bar{e}_t\sum_{s=1}^{t-\ell}e_s+(t-\ell)\bar{e}_t^2
-\bar{e}_t\sum_{s=t-\ell+1}^te_{s+\ell}\\
&\hspace{.5cm}-\tilde{e}_{T-t}\sum_{s=t-\ell+1}^te_s+\ell\bar{e}_t\tilde{e}_{T-t}-\tilde{e}_{T-t}\sum_{s=t+1}^{T-\ell}e_s-\tilde{e}_{T-t}\sum_{s=t+1}^{T-\ell}e_{s+\ell}+(T-\ell-t+1)\tilde{e}_{T-t}^2 \notag\\
&=\sum_{s=1}^{T-\ell}e_se_{s+\ell}+p_{\ell,t,1}+\ldots +p_{\ell,t,9}.\notag
\end{align}
(We used $\sum_\emptyset=0$.) We have similar expressions for $1\leq t<\ell$ and $ T-\ell\leq t\leq T$.\\

By the maximal inequality of M\'oricz et al.\ (1982) we have for all $x>0$ that
\begin{align*}
P\left\{\max_{\ell+1\leq t\leq T}\frac{1}{t^{1/2}}\left|\sum_{s=\ell+1}^{t}e_{s}\right|>x\right\}
&\leq P\left\{ \max_{1\leq i \leq \log T}e^{-(i-1)/2}\max_{e^{i-1}\leq t \leq e^s}\left|\sum_{s=\ell+1}^{t}e_{s}\right|>x\right\}\\
&\leq x^{-\bar{\nu}}\sum_{i=1}^{\log T}e^{-(i-1){\bar{\nu}/2}}E\max_{e^{i-1}\leq t \leq e^s}\left|\sum_{s=\ell+1}^{t}e_{s}\right|^{\bar{\nu}}\\
&\leq x^{-\bar{\nu}}\sum_{i=1}^{\log T}e^{-(i-1)\bar{\nu}/2}c_1e^{i{\bar{\nu}}/2}\\
&\leq c_2x^{-\bar{\nu}}\log T,
\end{align*}
with some constants $c_1$ and $c_2$. Hence, there is constant $c_3$ such that for all $\ell$ and $T$
\begin{align}\label{amo-1}
E\max_{\ell+1\leq t\leq T}\frac{1}{t^{1/2}}\left|\sum_{s=\ell+1}^{t}e_{s}\right|\leq c_3 (\log T)^{1/\bar{\nu}}.
\end{align}
We showed that Assumption \ref{aappr-a} yields
\beq\label{aed-10}
 \max_{1\leq t\leq T}\frac{1}{t^{1/2}}\left|\sum_{s=1}^te_s\right|=O_P((\log \log T)^{1/2}).
\eeq
Using \eqref{aed-10} we conclude
\begin{align*}
\max_{1\leq t\leq T}&\sum_{\ell=1}^{T-1}K(\ell/h)\left|\bar{e}_t\sum_{s=1}^{t-\ell}e_{s+\ell}\right|\\
&\leq  \max_{1\leq t\leq T}\frac{1}{t^{1/2}}\left|\sum_{s=1}^te_s
\right|\max_{1\leq t\leq T}\sum_{\ell=1}^{T-1}K(\ell/h)\left|\frac{1}{t^{1/2}}\sum_{s=1}^{t-\ell}e_{s+\ell}\right|\\
&=O_P((\log \log T)^{1/2})\sum_{\ell=1}^{T-1}K(\ell/h)\max_{\ell+1\leq t\leq T}\frac{1}{t^{1/2}}\left|\sum_{s=\ell+1}^{t}e_{s}\right|\\
&=O_P(h(\log T)^{1/\bar{\nu}}(\log \log T)^{1/2}),
\end{align*}
since by Assumption 5.1 and \eqref{amo-1} we have via Markov's inequality
$$
\sum_{\ell=1}^{T-1}K(\ell/h)\max_{\ell+1\leq t\leq T}\frac{1}{t^{1/2}}\left|\sum_{s=\ell+1}^{t}e_{s}\right|=O_P(h(\log T)^{1/\bar{\nu}}).
$$
Thus we have
\begin{align}\label{api-1}
\max_{1\leq t \leq T}\sum_{\ell=1}^{T-1}K(\ell/h)|p_{\ell,t,1}|=O_P(h(\log T)^{1/\bar{\nu}}(\log \log T)^{1/2}).
\end{align}
Similar arguments give
\begin{align}\label{api-2}
\max_{1\leq t \leq T}\sum_{\ell=1}^{T-1}K(\ell/h)|p_{\ell,t,i}|=O_P(h(\log T)^{1/\bar{\nu}}(\log \log T)^{1/2}),\quad i=2,4,5,7,8
\end{align}
and
\begin{align}\label{api-3}
\max_{1\leq t\leq T}\sum_{\ell=1}^{T-1}K(\ell/h)|p_{\ell,t,i}|=O_P(h\log \log T),\quad i=3,6,9.
\end{align}
It follows from \eqref{api-dec} and \eqref{api-1}--\eqref{api-3} that
\begin{align}\label{api-4}
\sum_{\ell=1}^{T-1}K(\ell/h)\frac{1}{T-\ell}\max_{\ell\leq t<T-\ell}\left|\sum_{s=1}^{T-\ell}e_se_{s+\ell}-\hat{\gamma}_{\ell,t}\right|=
O_P\left(\frac{h}{T}(\log T)^{1/\bar{\nu}}(\log \log T)^{1/2}\right).
\end{align}
Following the arguments leading to \eqref{api-4} one can verify that
\begin{align}\label{api-5}
\sum_{\ell=1}^{T-1}K(\ell/h)\frac{1}{T-\ell}\max_{0\leq t<\ell}\left|\sum_{s=1}^{T-\ell}e_se_{s+\ell}-\hat{\gamma}_{\ell,t}\right|=O_P\left(\frac{h}{T}
(\log T)^{1/\bar{\nu}}(\log \log T)^{1/2}\right)
\end{align}
and
\begin{align}\label{api-6}
\sum_{\ell=1}^TK(\ell/h)\frac{1}{T-\ell}
\max_{T-\ell\leq t \leq T}\left|\sum_{s=1}^{T-\ell}e_se_{s+\ell}-\hat{\gamma}_{\ell,t}\right|=O_P\left(\frac{h}{T}(\log T)^{1/\bar{\nu}}(\log \log T)^{1/2}\right).
\end{align}
The result of  Lemma \ref{api-lem}(i) is an immediate consequence of \eqref{api-4}--\eqref{api-6}.\\
The same arguments give the proof of Lemma \ref{api-lem}(ii).
\end{proof}

\medskip
{\it Proof of Theorem \ref{asi-th-2}.} The assumptions of the Theorem are stronger than those in Theorem 1 of Liu and Wu (2010),
\begin{align*}
\frac{1}{T}\sum_{t=1}^Te_t^2+2\sum_{\ell=1}^{T-1}K\left(\frac{\ell}{h}\right)\frac{1}{T-\ell}\sum_{s=1}^{T-\ell}e_se_{s+\ell}\stackrel{P}{\to}\sigma^2.
\end{align*}
Hence (5.1) and (5.1) follow immediately from Lemma \ref{api-lem}.
\qed


\begin{thebibliography}{99}



\bibitem{and91} Andrews, D.W.K.:  Heteroskedasticity and Autocorrelation Consistent Covariance Matrix Estimation. {\it Econometrica} {\bf 59}(1991), 817-858.
\bibitem{and93} Andrews, D.W.K.: Tests for parameter instability and structural change with unknown change point. {\it Econometrica} {\bf 61}(1993), 821--856.
\bibitem{and03} Andrews, D.W.K.: End-of-sample instability tests. {\it Econometrica} {\bf 71}(2003), 1661--1694.
\bibitem{ahh} Aue, A.,\  H\"ormann, S.,\  Horv\'ath, L. and   Hu\v{s}kov\'{a}, M.: Dependent functional linear models with applications to
monitoring structural change. {\it   Statistica Sinica} {\bf 24}(2014), 1043--1073.
\bibitem{ah} Aue, A.\ and Horv\'ath, L.: Structural breaks in time series. {\it Journal of Time Series Analysis} {\bf 34}(2013), 1--16.
\bibitem{b} Beach, C.M.: Changing income inequality: A distributional paradigm for Canada
{\it Canadian Journal of Economics} {\bf 49}(2016), 1229--1292.
\bibitem{bil} Billingsley, P.: {\it Convergence of Probability Measures.}
  Wiley, New York, 1968.
\bibitem{blw} Berkes, I.,\ Liu, W.\ and Wu, W.B.: Koml\'os--Major--Tusn\'ady approximation under dependence. {\it Annals of  Probability} {\bf 42}(2014), 794--817.
\bibitem{bradley} Bradley, R. C.: {\it Introduction to Strong Mixing Conditions.} Kendrick Press, Heber City, UT, 2007.
\bibitem{bls} Bureau of Labor Statistics, United States Department of Labor.: Major Sector Productivity and Costs.
Series IDs: PRS84006092, PRS84006093, PRS84006152, PRS84006153. url: http://www.bls.gov/ 
\bibitem{bus} Busetti, F.: On detecting end-of-sample instabilities, in: S.J. Koopman and N. Shephard eds, Unobserved Components and Time Series Econometrics, Oxford University Press, 2015.
\bibitem{mcs-1} Cs\"org\H{o}, M.: Some R\'enyi type limit theorems for empirical distribution functions. {\it Annals of Mathematical Statistics} {\bf 36}(1965), 322-326.
\bibitem{csh-0} Cs\"org\H{o}, M.\ and  Horv\'ath, L.:  R\'enyi--type empirical processes. {\it Journal of Multivariate Analysis} {\bf 41}(1992), 338--358.
\bibitem{csh-1} Cs\"org\H{o}, M.\ and  Horv\'ath, L.:  {\it Weighted Approximations  in
Probability and Statistics.} Wiley, 1993.
\bibitem{csh-2} Cs\"org\H{o}, M.\ and  Horv\'ath, L.: {\it Limit Theorems in Change--point Analysis.} Wiley,  1997.
\bibitem{csr} Cs\"org\H{o}, M.\ and R\'ev\'esz, P.: {\it Strong Approximations in Probability and Statistics}. Academic Press, New York, 1981.
\bibitem{darer} Darling, D.A. \and Erd\H{o}s, P.: A limit theorem for the maximum of normalized sums of independent random variables. {\it Duke Mathematical Journal} {\bf 23}(1956), 143-155.
\bibitem{dhy} Davis, R.A.,\ Huang, D.\ and Yao, Y.-C.: Testing for a change in the parameter values and order of an autoregressive model. {\it Annals of Statistics} {\bf 23}(1995), 282--304.
\bibitem{ghk} G\'orecki, T.,\ Horv\'ath, L.\ and Kokoszka, P.: {\it Change point detection in heteroscedastic time series.} Preprint.
http://arxiv.org/abs/1612.02794
\bibitem{hahe} Hall, P.\ and Heyde, C.C.: {\it Martingale Limit Theory and Its Application.}  Academic Press, 1980.



\bibitem{kue} Kuelbs, J.\ and Philipp, W.:  Almost sure invariance principles for partial sums of mixing B--valued random variables.
{\it Annals of  Probability} {\bf 8}(1980) 1003--1036.
\bibitem{ling} Ling, S.: Testing for change points in time series models and limiting theorems for NED sequences. {\it  Annals of Statistics} {\bf 35}(2007), 1213--1237.
\bibitem{lw} Liu, W.\ and Wu, W.B.: Asymptotics of spectral density estimates. {\it Econometric Theory} {\bf 26}(2010), 1218--1245.
\bibitem{mss} M\'oricz, F.,\ Serfling, R.J.\ and Stout, W.F.: Moment and probability bounds with quasi--superadditive structure for the maximum partial sum. {\it Annals of probability} {\bf 10}(1982), 1032--1040.

\bibitem{pvj}   Perron, P. \and Vogelsang, T.J.: Testing for a Unit Root in a Time Series with a Shift in Mean: Corrections and Extensions, {\it Journal of Business and Economic Statistics}, {\bf 10} (1992), 467-470.
\bibitem{pri} Priestley, M.B. {\it Spectral Analysis and Time Series I.} Academic Press, 1981.
\bibitem{reich} Reich, R. B.: {\it Saving Capitalism for the Many, Not the Few}.  Alfred A. Knopf, New York, 2015.

\bibitem{renyi} R\'enyi, A.: On the theory of order statistics. {\it Acta Mathematica Academiae Scientiarum Hungaricae} {\bf 4}(1953), 191--231.


\bibitem{vtj} Vogelsang, T.J.:  Sources of nonmonotonic power when testing for a shift in
mean of a dynamic time series. {\it Journal of Econometrics} {\bf 88} (1999), 283-299.







\begin{comment}
\bibitem{ahp} Antoch, J.,\ Hu\v{s}kov\'a, M.,\ and Pr\'a\v{s}kov\'a, Z.:  Effect of dependence on statistics for determination of change.
{\it Journal of Statistical Planning and Inference} {\bf 60}(1997), 291--310.
\bibitem{ahhr} Aue, A.,\ H\"{o}rmann, S.,\ Horv\'ath, L., and Reimherr, M.:  Break detection in the
covariance structure of multivariate time series models. {\it The Annals of Statistics} {\bf 37}(2009), 4046--4087.
\bibitem{ah} Aue, A.\ and Horv\'ath, L.: Structural breaks in time series. {\it Journal of Time Series Analysis} {\bf 34}(2013), 1--16.
\bibitem{bgh} Berkes, I.,\ Gombay, E.\ and Horv\'ath, L.: Testing for changes in the covariance
structure of linear rocesses. {\it Journal of Statistical Planning and Inference} {\bf 139}(2009), 2044--2063.
\bibitem{bhr} Berkes, I., Horv\'ath, L.\ and  Rice, G.:  Weak invariance principles for sums of dependent random functions. {\it Stochastic Processes and Their Applications} {\bf 123}(2013), 385--403.

\bibitem{brei} Breiman, L.: {\it Probability.} Addison Wesley, 1968.
\bibitem{bt} Busetti, F.\ and Taylor, A.R.M.: Test of stationarity against a change in persistence. {\it Journal of Econometrics} {\bf 123}(2004), 33--66.
\bibitem{chlt} Cavaliere, G.,\ Harvey, D.I.,\ Leybourne, S.J.\  and Taylor, A.M.R.: Testing for unit roots in the presence of a possible break in trend and nonstationarity volatility. {\it Econometric Theory} {\bf 27}(2011), 957--991.
\bibitem{ct} Cavaliere, G.\ and Taylor, A.M.R.: Testing for a change in persistence in the presence of non--stationary volatility. {\it Journal of Econometrics} {\bf 147}(2008), 84--98.
\bibitem{csh97} Cs\"org\H{o}, M.\ and Horv\'ath, L.: {\it Limit Theorems in Change--Point Analysis.} Wiley, 1997.
\bibitem{csr} Cs\"org\H{o}, M.\ and R\'ev\'esz, P.: {\it Strong Approximations in Probability and Statistics.} Academic Press,  1981.
\bibitem{dgp} Dalla, V.,\ Giraitis, L.\ and Phillips, P.C.B.: Testing mean stability of heteroskedastic time series. Preprint, 2015.
\bibitem{dhy} Davis, R. A.,\ Huang, D.\ and Yao, Y--C.:   Testing for a change in the parameter values and order of an autoregressive model.
{\it The Annals of Statistics} {\bf 23}(1995), 282--304.
\bibitem{dp} Deng, A.\ and Perron, P.:  A Non--local perspective on the power properties of the
CUSUM and CUSUM of squares tests for structural change. {\it Journal of Econometrics} {\bf  142}(2008), 212--240.
\bibitem{ds} Dunford, N.\ and Schwartz, J.T.: {\it Linear Operators: General Theory (Part 1). } Springer, 1988.
\bibitem{francqzako} Francq, C.\ and  Zakoian, J.-M.: {\it GARCH Models.}  Wiley,  2010.
\bibitem{gh} Gombay, E.\  and    Horv\'ath, L.:    An  application  of  the  maximum  likelihood  test  to  the
change-point  problem. {\it Stochastic  Processes  and  Their  Applications}  {\bf 50}(1994), 161–-171.
1994
\bibitem{hans} Hansen, B.E.: Tests for parameter instability in regression with $I(1)$ processes. {\it Journal of Business and Economics Statistics} {\bf 10}(1992), 321--335.
\bibitem{hlt} Harvey, D.I.,\ Leybourne, S.J.\ and Taylor, A.M.R.: Modified tests for a change in persistence. {\it Journal of Econometrics} {\bf 134}(2006), 441--469.
\bibitem{hew} Hewitt, E.\ and Stromberg, K.: {\it Real and Abstract Analysis.} Springer--Verlag, 1969.
\bibitem{bc} Horv\'ath, L.\ and Kokoszka, P.: {\it  Inference for Functional Data with Applications.} Springer, New York, 2012.
\bibitem{it} Incl\'an, C.\ and Tiao, G. C.:  Use of cumulative sums of squares for retrospective detection of change of variance.
{\it Journal of the American Statistical Association}  {\bf 89}(1994), 913--923
\bibitem{jeng} Jeng, J.--L.: {\it Analyzing Event Statistics in Corporate Finance.} Palgrave and MacMillen, 2015.
\bibitem{jirak} Jirak, M.: On weak invariance principals for sums of dependent random functionals. {\it Statistics \& Probability Letters} {\bf 83}(2013), 2291--2296.
\bibitem{kmt-1} Koml\'os, J.,\ Major, P.\ and Tusn\'ady, G.:  An approximation of partial sums of independent rv’s and the sample df.\ I.
  {\it Zeitschrift f\"ur Wahrscheinlichkeitstheorie und  Verwandte Gebiete} {\bf 32}(1975), 111–-131.
\bibitem{kmt-2} Koml\'os, J.,\ Major, P.\ and Tusn\'ady, G.:  An approximation of partial sums of independent rv’s and the sample df.\ II. {\it Zeitschrift f\"ur Wahrscheinlichkeitstheorie und  Verwandte Gebiete} {\bf 34}(1976), 33–-58.
\bibitem{lp} Lee, S.\ and Park, S.: The Cusum of squares test for scale changes in infinite order moving average processes.
{\it Scandinavian Journal of Statistics} {\bf 28}(2001), 625--644.

\bibitem{shw} Shorack, G.R.\ and Wellner, J.A.: {\it Empirical Processes with Applications to Statistics.} Wiley, 1986.
\bibitem{wkd} Wied, D.,\ Kr\"amer, W.\ and  Dehling, H.:  Testing for a change in correlation at an unknown point in time using an extended functional delta method. {\it Econometric Theory} {\bf 28 }(2012), 570–-589.
\bibitem{wzb} Wied, D.,\ Ziggel, D.\ and Berens, T.:  On the application of new tests for structural changes on global minimum-variance portfolios. {\it Statistical Papers}  {\bf 54 }(2013),  955–-975.
\end{omment}
\end{thebibliography}
\end{document}